\documentclass[a4paper,12 pt]{article}
\usepackage[utf8]{inputenc}
\usepackage{amsmath}
\usepackage{amssymb}
\usepackage{latexsym}
\usepackage{amsthm}
\usepackage{amsfonts}
\usepackage{graphicx}
\usepackage{color}
\usepackage{shuffle}
\usepackage{graphicx}
\usepackage{xcolor}

\newtheorem{theorem}{Theorem}

\newtheorem{proposition}{Proposition}
\newtheorem{remark}{Remark}
\newtheorem{corollary}{Corollary}

\newtheorem{example}{Example}
\allowdisplaybreaks

\def \a{\alpha}
\def \z{\zeta}

\def\r{\rightarrow}

\def\c{\mathbb{C}^r}

\usepackage[a4paper,top=3cm,bottom=2cm,left=1.5cm,right=1.5cm,marginparwidth=1cm]{geometry}
\title{Meromorphic continuation of a $q$-analogue of multiple zeta function }
\author{$\text{Nita Tamang}^a, \text{Pitu Sarkar}^b$ \\ $\text{Department of Mathematics, University of North Bengal}^{a,b}$\\
	West Bengal, India 734013\\ $\text{nita\_math@nbu.ac.in}^a$,~$\text{pitucob2016@gmail.com}^b$  }
\date{}

\begin{document}
	\maketitle
	
	\begin{abstract}
		In this paper, we obtain the meromorphic continuation of a $q$-analogue of the multiple zeta function using an elementary formula called the translation formula.  We then obtain the matrix representation of the translation formula using which, we locate the poles of the function and find the corresponding residues. While locating the poles, we also obtain the inverse of an infinite triangular matrix in a particular case.  
	\end{abstract}
	\hspace{1cm}Keywords: $q$-multiple zeta function; Meromorphic continuation; Pole; Residue.\\ 
	\hspace*{1cm}Mathematics Subject Classification 2020: 11M32.
	\section{Introduction}
	Let $0 < q<1$. For any positive integer $k$, define its $q$-analog as $[k]=[k]_q = \dfrac{1-
		q^k}{1-q}$. For $r \geq 1$, consider the subset $U_r$ of $\mathbb{C}^r$ consisting of all $r$-tuples $(s_1, s_2, \dots ,s_r)$ of complex numbers satisfying the conditions
	$$\text{Re}(s_1+ \dots +s_j)>0 \hspace{1cm}\text{for}~ j=1, 2, \dots , r.$$ Then $U_r$ is a convex open subset of $\mathbb{C}^r$. The $q$-multiple zeta function of depth $r$ on $U_r$ is defined by
	\begin{align}
		\label{1}
		&\z_q(s_1, \dots, s_r)=\sum_{k_1>\dots >k_r\geq 1}\dfrac{q^{k_1s_1+ \dots +k_rs_r}}{[k_1]^{s_1} \dots [k_r]^{s_r }}.
	\end{align}  
This function is a $q$-analogue of the (Euler–Zagier) multiple zeta function, defined by
\begin{align}
	\label{eq2}
	\zeta ( s_1,\dots,s_r)=\sum_{k_1>\dots >k_r\geq 1}\dfrac{1}{k_1^{s_1}\dots k_r^{s_r}},
\end{align}
where $s_1, \dots,s_r$ are complex variables satisfying $ \text{Re}(s_1 + \dots +s_j)>j $ for all $j =
1, \dots ,r$. If $(s_1 , \dots ,s_r) \in \c $, with $ \text{Re}(s_1 + \dots +s_j)>j $, for all $j =
1, \dots ,r$, then the series \eqref{1} converges to $\zeta( s_1,\dots, s_r)$ as $q \r 1.$ 
In \cite{5},  Zudilin studied the function \eqref{1} at positive integer arguments as an
appropriate $q$-extension of multiple zeta values. It essentially coincides with the function defined by Schlesinger
in \cite{6} as well. Hence it is called the SZ model of the $q$-multiple zeta function and it is denoted by  $\z_q^{\text{SZ}}(s_1, \dots, s_r)$. Throughout this article, by $q$-multiple zeta function, we mean $\z_q^{\text{SZ}}(s_1, \dots, s_r)$ and use $\z_q(s_1, \dots, s_r)$ to denote it.  Many authors have studied and shown through different approaches that the multiple zeta function of depth $r$ can be meromorphically continued to the whole complex plane $\c$ \cite{3,4, 8}. One can ask a natural question here whether such a continuation of the function $(1)$ happens to exist or not. In \cite{1}, J. Zhao studied a more general function of $2r$ complex variables $s_1,...,s_r,t_1,...,t_r \in \mathbb{C}$, which is defined by
\begin{align*}
	&f_q(s_1,...,s_r;t_1,...,t_r)=\sum_{k_1>\dots >k_r\geq 1}\dfrac{q^{k_1t_1+ \dots +k_rt_r}}{[k_1]^{s_1} \dots [k_r]^{s_r}},
\end{align*}
and converges if $\text{Re}(t_1 + \dots + t_j)>0$, for all $j = 1, \dots ,r$. He extended $f_q$ to a meromorphic function on $\mathbb{C}^{2r}$ with explicit poles using the binomial theorem. The case $r=1$ was first studied by Kaneko et al. in \cite{2}.\\

In \cite{7}, Mehta, Saha and Viswanadham obtain the meromorphic continuation of the multiple zeta function of depth $r$ using a translation formula which generalizes the formula for the Riemann zeta function given by Ramanujan\cite{5.14}.
Inspired by their work, we obtain a similar kind of translation formula and use it to obtain the meromorphic continuation of the $q$-multiple zeta function in Section $2$.  Then we express our translation formula in terms of certain infinite matrices in Section 3.  Using the matrix representation of the translation formula, we locate the poles of the $q$-multiple zeta function and express their residues as entries of these matrices. Finding these residues is a new addition to the literature. Finally, in Section 5, we give some observations on the meromorphic continuation of another model of the $q$-multiple zeta function.

\section{Meromorphic continuation of $q$-multiple zeta function}
In this section, we prove the normal convergence of the $q$-multiple zeta function on any compact subset of $U_r$ and obtain the translation formula. We then establish its meromorphic continuation. For that, we first prove the following result:
\begin{proposition}
	\label{p1}
	Let $r$ be a positive integer and $(\alpha_1, \dots , \alpha_r)$ be an $r$-tuple of real numbers in $U_r$. Then the family of functions 
	\begin{align}
		\label{5.5}
		{\Bigg(\dfrac{q^{k_1s_1+ \dots +k_rs_r}}{[k_1]^{s_1} \dots [k_r]^{s_r }}\Bigg)}_{k_1>\dots >k_r\geq 1}	
	\end{align}
	is normally convergent on the set $\mathcal{D}(\a_1, \dots , \a_r)=\{(s_1, \dots, s_r) \in \c: \text{Re}(s_j)\geq \a_j, 1 \leq j \leq r\}.$ 
\end{proposition}
\begin{proof}
	Since $0<q<1$, we have $\left|\dfrac{q^{k_i}}{[k_i]}\right|=\left|\dfrac{q^{k_i}}{1+q+\dots+q^{k_i-1}}\right| \leq 1$, for all $i=1, 2, \dots r$. 
	Also, for $(s_1, \dots, s_r)\in \mathcal{D}(\a_1, \dots , \a_r)$, we have $\text{Re}(s_j)\geq \a_j$, where $1 \leq j \leq r$. 
	Therefore for $(s_1, \dots, s_r)\in \mathcal{D}(\a_1, \dots , \a_r)$,  $\left|\bigg(\dfrac{q^{k_i}}{[k_i]} \bigg)\right|^{s_i} \leq \left|\bigg(\dfrac{q^{k_i}}{[k_i]} \bigg)\right|^{\a_i} $, for all $i=1, \dots, r$. 
	Thus for each sequence of integers $k_1>\dots >k_r\geq 1$, we have
	\begin{align*}
		\bigg\lVert \dfrac{q^{k_1s_1+ \dots +k_rs_r}}{[k_1]^{s_1} \dots [k_r]^{s_r }}\bigg\rVert_{\mathcal{D}_{(\a_1, \dots, \a_r)}}=\dfrac{q^{k_1\a_1+ \dots +k_r\a_r}}{[k_1]^{\a_1} \dots [k_r]^{\a_r }}.
	\end{align*} 
	Now the family 
	\begin{align*}
		{\Bigg(\dfrac{q^{k_1\a_1+ \dots +k_r\a_r}}{[k_1]^{\alpha_1} \dots [k_r]^{\alpha_r }}\Bigg)}_{k_1>\dots >k_r\geq 1}	
	\end{align*}
	is convergent by \cite[Proposition 2.2]{1}, since $(\alpha_1, \dots , \alpha_r) \in U_r$. This implies that the family \eqref{5.5} is normally convergent on  $ \mathcal{D}(\a_1, \dots , \a_r)$.
\end{proof}
\begin{corollary}
	\label{c1}
	For a positive integer $r$, the series on the R.H.S. of \eqref{1} converges normally on any compact subset of $U_r$.
\end{corollary}
\begin{proof}
	By the definition of $U_r$, any point $(s_1, \dots, s_r) \in U_r,$  has a neighborhood of the form 
	\begin{align*}
		\mathcal{D}(\a_1, \dots , \a_r)=\{(s_1, \dots, s_r) \in \c: \text{Re}(s_j)\geq \a_j, 1 \leq j \leq r\},\end{align*}
	where $(\alpha_1, \dots , \alpha_r)$ is an $r$-tuple of real numbers in $U_r$. By Proposition \ref{p1}, the series converges normally on this neighborhood. Consequently, the series is normally convergent on every compact subset of $U_r$. 
\end{proof}
Now consider $$A= {\Bigg(\frac{[n]}{q^{n-1}}\Bigg)}^{-s} - {\Bigg(\frac{[n]}{q^n}+1\Bigg)}^{-s},$$	
where $n$ is a positive integer and $s$ is a complex number. 
Then by the Taylor series expansion, 
\begin{align*}
	&A={\bigg(\frac{[n]}{q^{n-1}}\bigg)}^{-s}-{\bigg(\frac{[n]}{q^n}\bigg)}^{-s}\sum_{k \geq 0} {(-1)}^k \frac{s(s +1) \cdots (s + k -1)}{k!}{\bigg(\frac{q^n}{[n]}\bigg)}^{k}
	\nonumber \\
	&={\bigg(\frac{[n]}{q^{n-1}}\bigg)}^{-s}-\sum_{k \geq 0} {(-1)}^k \frac{s(s +1) \cdots (s + k -1)}{k!}{\bigg(\frac{q^n}{[n]}\bigg)}^{s+k}\nonumber \\
	&={\bigg(\frac{q^{n}}{[n]}\bigg)}^{s}(q^{-s}-1)+\sum_{k \geq 0} {(-1)}^{k} \frac{s(s +1) \cdots (s + k )}{(k+1)!}{\bigg(\frac{q^n}{[n]}\bigg)}^{s+k+1}.
\end{align*}
Thus \begin{align}
	\label{eq4}
	{\Bigg(\frac{[n]}{q^{n-1}}\Bigg)}^{-s} - {\Bigg(\frac{[n]}{q^n}+1\Bigg)}^{-s}={\bigg(\frac{q^{n}}{[n]}\bigg)}^{s}(q^{-s}-1)+\sum_{k \geq 0} {(-1)}^{k} \frac{s(s +1) \cdots (s + k )}{(k+1)!}{\bigg(\frac{q^n}{[n]}\bigg)}^{s+k+1}.
\end{align}

\begin{proposition}
	\label{pro1}
	Let $r \geq 2$ be an integer. The family of functions on $U_r$,
	\begin{align*}
		& {\bigg( {(-1)}^{k} \frac{s_1(s_1 +1) \cdots (s_1 + k )}{(k+1)!}\dfrac{q^{{n_1}{(s_1+k+1)}}q^{n_2s_2+n_3s_3+ \dots +n_rs_r}}{{[n_1]}^{s_1+k+1}[n_2]^{s_2} [n_3]^{s_3}\dots [n_r]^{s_r}}\bigg)}_{n_1 >n_2 > n_3 > \dots >n_r \geq 1,~ k \geq 0} 
	\end{align*}
	are normally summable on any compact subset of $U_r$.
\end{proposition}
\begin{proof}
	Let $S$ be a compact subset of $U_r$ and $b$ denote the supremum of $|s_1|$ in $S$. Consider
	\begin{align*}
		& f_1={(-1)}^{k}  \frac{s_1(s_1 +1) \cdots (s_1 + k )}{(k+1)!}\dfrac{q^{{n_1}{(s_1+k+1)}}q^{n_2s_2+n_3s_3+ \dots +n_rs_r}}{{[n_1]}^{s_1+k+1}[n_2]^{s_2} [n_3]^{s_3}\dots [n_r]^{s_r}}.
	\end{align*}	
	Then for any $r$-tuple $(n_1, n_2, \dots , n_r)$ of positive integers with $n_1 >n_2 > \dots >n_r \geq 1$ and any integer $k \geq 0$, we have
	\begin{align*}
		||f_1||_S \leq \frac{b(b+1) \cdots (b+k)}{(k+1)!}\left|\dfrac{q^{{n_1}(k+1)}}{{[n_1]}^{k+1}}\right|\bigg\lVert\dfrac{q^{n_1s_1+n_2s_2+ \dots +n_rs_r}}{{[n_1]}^{s_1}[n_2]^{s_2} \dots [n_r]^{s_r}}\bigg\rVert_S, \hspace{1cm}(\text{since}~ q < 1 ).
	\end{align*} 
	We can see that, $\left|\dfrac{q^{{n_1}(k+1)}}{{[n_1]}^{(k+1)}}\right| < c^{(k+1)}$, for some $c <1$.
	The family $$\Bigg(\bigg\lVert\dfrac{q^{n_1s_1+n_2s_2+ \dots +n_rs_r}}{{[n_1]}^{s_1}[n_2]^{s_2} \dots [n_r]^{s_r}} \bigg\rVert_S\Bigg)_{n_1 >n_2 > \dots >n_r \geq 1}$$ is summable by Corollary \ref{c1} and the series $\sum_{k \geq 0}\frac{b(b+1) \cdots (b+k)}{(k+1)!}c^{(k+1)}$ converges. This implies that the family $\big(||f_1||_S\big)_{n_1 >n_2 > \dots >n_r \geq 1,~ k \geq 0}$ is normally summable on $S$, which is our desired result. 
\end{proof}
Now taking summation over $n$  from $1$ to $\infty$ on both sides of \eqref{eq4}, we get
\begin{align}
	\label{5.8}
	&1= \sum_{n \geq 1} \bigg[{\bigg(\frac{q^{n}}{[n]}\bigg)}^{s}(q^{-s}-1)+\sum_{k \geq 0} {(-1)}^{k} \frac{s(s +1) \cdots (s + k )}{(k+1)!}{\bigg(\frac{q^n}{[n]}\bigg)}^{s+k+1}\bigg].
\end{align}
For Re$(s) >1$,	\eqref{5.8} gives the following translation formula for the $q$-analogue of the Riemann zeta function
\begin{align*}
	&1= (q^{-s}-1)\z_q(s) +\sum_{k \geq 0} {(-1)}^{k} \frac{s(s +1) \cdots (s + k )}{(k+1)!}\z_q(s+k+1) \nonumber \\
	& \implies q^{s}=(1-q^{s})\z_q(s)+ q^{s}\sum_{k \geq 0} {(-1)}^{k} \frac{s(s +1) \cdots (s + k )}{(k+1)!}\z_q(s+k+1).
\end{align*}
Letting $q \rightarrow 1$, we obtain a translation formula for the Riemann zeta function  with alternating signs as follows:
\begin{align*}
	1=\sum_{k \geq 0} {(-1)}^{k} \frac{s(s +1) \cdots (s + k )}{(k+1)!}\z(s+k+1).
\end{align*}
This formula is similar to that given by Ramanujan in \cite{5.14}.
Now \eqref{eq4} can be written as
\begin{align}
	\label{eq5.454}
	{\Bigg(\frac{[n]}{q^{n-1}}\Bigg)}^{-s} - {\Bigg(\frac{[n+1]}{q^n}\Bigg)}^{-s}={\bigg(\frac{q^{n}}{[n]}\bigg)}^{s}(q^{-s}-1)+\sum_{k \geq 0} {(-1)}^{k} \frac{s(s +1) \cdots (s + k )}{(k+1)!}{\bigg(\frac{q^n}{[n]}\bigg)}^{s+k+1}.
\end{align}
Taking $n=n_{1}$ and summing over $n_1$ from $n_2+1$ to $\infty$ on both sides of \eqref{eq5.454}, we get
\begin{align}
	&{\bigg(\frac{[n_2+1]}{q^{n_2}}\bigg)}^{-s_1}=\sum_{n_1=n_2+1}^{\infty}\bigg[{\bigg(\frac{q^{n_1 }}{[n_1]}\bigg)}^{s_1}(q^{-s_1}-1)+\sum_{k \geq 0} {(-1)}^{k} \frac{s_1(s_1 +1) \cdots (s_1 + k )}{(k+1)!}{\bigg(\frac{q^{n_1}}{[n_1]}\bigg)}^{s_1+k+1}\bigg]\nonumber\\
	&\implies q^{n_2s_1}{\big([n_2]+q^{n_2}\big)}^{-s_1}=\sum_{n_1=n_2+1}^{\infty}\bigg[{\bigg(\frac{q^{n_1 }}{[n_1]}\bigg)}^{s_1}(q^{-s_1}-1)\nonumber\\
	&\hspace{5cm} +\sum_{k \geq 0} {(-1)}^{k} \frac{s_1(s_1 +1) \cdots (s_1 + k )}{(k+1)!}{\bigg(\frac{q^{n_1}}{[n_1]}\bigg)}^{s_1+k+1}\bigg].\nonumber 
\end{align}
Expanding L.H.S of the above equation as a Taylor series in $\frac{q^{n_2}}{[n_2]}$, we get 
\begin{align}
	\label{eq7}
	& \sum_{k \geq 0} {(-1)}^{k} \frac{s_1 (s_1 +1)\cdots (s_1 + k -1)}{k!}{\bigg(\frac{q^{n_2}}{[n_2]}\bigg)}^{s_1+k} \nonumber \\ &=\sum_{n_1=n_2+1}^{\infty}\bigg[{\bigg(\frac{q^{n_1 }}{[n_1]}\bigg)}^{s_1}(1-q^{-s_1})+\sum_{k \geq 0} {(-1)}^{k} \frac{s_1(s_1 +1)\cdots (s_1 + k )}{(k+1)!}{\bigg(\frac{q^{n_1}}{[n_1]}\bigg)}^{s_1+k+1}\bigg].
\end{align}
Multiplying both sides of \eqref{eq7} by $ \prod_{i=2}^{r}{\Big(\frac{q^{n_i}}{[n_i]}\Big)}^{s_i}$ and taking summation over $n_r$ from $1$ to $\infty$, over $n_{r-1}$ from $n_r+1$ to $\infty$ and so on up to over $n_2$ from $n_3+1$ to $\infty$, we get
\begin{align}
	\label{eq8}
	&\sum_{n_2 > n_3 > \dots >n_r \geq 1}\sum_{k \geq 0} {(-1)}^{k} \frac{s_1 (s_1 +1)\cdots (s_1 + k -1)}{k!}\dfrac{q^{{n_2}{(s_1+s_2+k)}}q^{n_3s_3+ \dots +n_rs_r}}{[n_2]^{s_1+s_2+k} [n_3]^{s_3}\dots [n_r]^{s_r}} \nonumber \\
	&= (q^{-s_1}-1)\sum_{n_1 >n_2 > \dots >n_r \geq 1}\dfrac{q^{n_1s_1+n_2s_2+ \dots +n_rs_r}}{{[n_1]}^{s_1}[n_2]^{s_2} \dots [n_r]^{s_r}} \nonumber \\
	& \ \ + \sum_{n_1 >n_2 > n_3 > \dots >n_r \geq 1}\sum_{k \geq 0} {(-1)}^{k} \frac{s_1 (s_1 +1)\cdots (s_1 + k )}{(k+1)!}\dfrac{q^{{n_1}{(s_1+k+1)}}q^{{n_2}{s_2}+n_3s_3+ \dots +n_rs_r}}{{[n_1]}^{s_1+k+1}[n_2]^{s_2} [n_3]^{s_3}\dots [n_r]^{s_r}}.
\end{align}

By Proposition \ref{pro1}, we can easily see that all the series of functions appearing in \eqref{eq8} are normally convergent on any compact subset of $U_r$. This immediately gives the next theorem.	
\begin{theorem}
	\label{th1}
	Let $r  \geq 2$ be any positive integer. Then for any $(s_1, s_2, \dots ,s_r)\in U_r$, we have 
	\begin{align}
		\label{eq12}
		&\sum_{k \geq 0} {(-1)}^{k} \frac{s_1 (s_1 +1)\cdots (s_1 + k -1)}{k!}\z_q(s_1+s_2+k,s_3, \dots , s_r)\nonumber \\
		& =\sum_{k \geq 0}{(-1)}^{k} \frac{s_1(s_1 +1) \cdots (s_1 + k )}{(k+1)!}\z_q(s_1+k+1, s_2 , \dots , s_r)+( q^{-s_1}-1)\z_q(s_1,s_2, \dots , s_r),
	\end{align}
	where both the series appearing in \eqref{eq12} converge normally on any compact subset of $U_r$. 	
\end{theorem}
\begin{proof}
	The proof is simple and hence omitted.
\end{proof}
Eq. \eqref{eq12} is the translation formula for $q$-multiple zeta function of depth $r$.
Using this translation formula, we show that the $q$-multiple zeta function of depth $r$ can be
extended meromorphically to the whole complex plane $\c$.	
\begin{theorem}
	\label{th2}
	Let $r \geq 2$ be any positive integer. Then the $q$-multiple zeta function of depth $r$ extends to a meromorphic function on $\mathbb{C}^r$ satisfying 
	\begin{align}
		\label{eq14}
		&\sum_{k \geq 0} {(-1)}^{k} \frac{s_1 (s_1 +1)\cdots (s_1 + k -1)}{k!}\z_q(s_1+s_2+k,s_3, \dots , s_r)\nonumber \\
		& =\sum_{k \geq 0}{(-1)}^{k} \frac{s_1(s_1 +1) \cdots (s_1 + k )}{(k+1)!}\z_q(s_1+k+1, s_2 , \dots , s_r)+( q^{-s_1}-1)\z_q(s_1,s_2, \dots , s_r),
	\end{align}
	where both the series appearing in \eqref{eq14} converge normally on any compact subset of $U_r$. 	 	
\end{theorem}

\begin{proof}
	We prove by induction on $r$. 
	When $r =2$, the left hand side of \eqref{eq14} is a meromorphic function because it is a normally convergent series of the $q$-analogue of the Riemann zeta function, which is meromorphic on $\mathbb{C}$ by \cite{2}.
	Then for $r \geq 3$, the left
	hand side of \eqref{eq14} is a meromorphic function by the induction hypothesis.
	For any $N \geq 0$, let $U_r(N)$ denote the open subset of $\mathbb{C}^r$ defined by
	$$\text{Re}(s_1+ \dots +s_j)>-N, \hspace{1cm}\text{for}~ j=1, 2, \dots , r.$$
	We will prove by induction on $N$ that the $q$-multiple zeta function of depth $r$ extends to a meromorphic function on $U_r(N)$. 
	Since ${(U_r(N))}_{N \geq  0}$ is an open covering of $\mathbb{C}^r$, Theorem \ref{th2} will follow. 
	For $N=0$, we have $U_r(N)=U_r$. So by Theorem \ref{th1}, it is true.
	Assume now $N \geq 1$, and the theorem is true for $N-1$, i.e., any $q$-multiple zeta function of depth $r$ can be extended to a meromorphic function on $U_r(N-1)$.  For $k \geq 0$,  $(s_1+k, s_2, \dots , s_r) \in U_r(N) \implies (s_1+k+1, s_2, \dots , s_r) \in U_r(N-1)$, because $\text{Re}(s_1+k+s_2+\dots +s_j)>-N \implies \text{Re}(s_1+k+1+s_2+\dots +s_j)>-N+1$, for $j=1, 2, \dots , r$. Then on the R.H.S. of \eqref{eq14}, except possibly $\z_q(s_1, s_2, \dots , s_r)$,  all the terms are meromorphic on $U_r(N)$ and  corresponding to $k \geq N-1$, all are holomorphic on $U_r(N)$.
	In order to show that the series
	\begin{align}
		\label{eq177} 
		\sum_{k \geq N-1}{(-1)}^{k} \frac{s_1(s_1+1) \dots (s_1 + k)}{(k+1)!}\z_q(s_1+k+1, s_2 , \dots , s_r)
	\end{align}
	is a holomorphic function on $U_r(N)$; it is enough to show that it is normally convergent on any compact subset $S$ of $U_r(N)$.
	As in the proof of Theorem 1, let $b$ be the supremum of $|s_1|$ in $S$. Then for any $r$ tuple $(n_1, n_2, \dots , n_r)$ of positive integers with $n_1 >n_2 > \dots >n_r \geq 1$, and any integer $k \geq N-1$, 
	\begin{align*}
		&\Bigg\|{(-1)}^{k}  \frac{s_1(s_1 +1) \cdots (s_1 + k )}{(k+1)!}\dfrac{q^{{n_1}{(s_1+k+1)}}q^{n_2s_2+n_3s_3+ \dots +n_rs_r}}{{[n_1]}^{s_1+k+1}[n_2]^{s_2} [n_3]^{s_3}\dots [n_r]^{s_r}}\Bigg\|_S
	\end{align*}
	is bounded above by 
	\begin{align*}
		&	\frac{b(b+1) \cdots (b+k)}{c^{k-N}(k+1)!}\Bigg\|\dfrac{q^{{n_1}{(s_1+N)}}q^{n_2s_2+n_3s_3+ \dots +n_rs_r}}{{[n_1]}^{s_1+N }[n_2]^{s_2} [n_3]^{s_3}\dots [n_r]^{s_r}}\Bigg\|_S,
	\end{align*}
	where $c>1,$
	since $q < 1 \implies \dfrac{q^{n_1}}{\frac{1-q^{n_1}}{1-q}}=\dfrac{q^{n_1}}{[n_1]} < 1 \text{ and } k+1 \geq N \implies {\bigg(\dfrac{q^{n_1}}{[n_1]}\bigg)}^{k+1}\leq  {\bigg(\dfrac{q^{n_1}}{[n_1]}\bigg)}^{N}.$ 
	Now the family $\Bigg( \bigg\|\dfrac{q^{{n_1}{(s_1+N)}}q^{n_2s_2+n_3s_3+ \dots +n_rs_r}}{{[n_1]}^{s_1+N }[n_2]^{s_2} [n_3]^{s_3}\dots [n_r]^{s_r}}\bigg\|_S\Bigg)_{n_1 >n_2> \dots >n_r\geq 1}$ is summable as $(s_1+N, s_2, \dots , s_r) \in U_r$. Also the series $\sum_{k\geq N-1} \frac{b(b+1) \cdots (b+k)}{c^{k-N}(k+1)!}$ is convergent.
	Hence the series \eqref{eq177} is holomorphic on $U_r(N)$ 
	and thus $\z_q(s_1, \dots, s_r)$ is meromorphic on  $U_r(N)$. This completes the proof.   	
\end{proof}

\section{ Matrix representation of the translation formula.}
As mentioned in the introduction, we now present a matrix formulation of the translation formula \eqref{eq14}.  Let $M=(m_{ij})_{\mathbb{N}\times \mathbb{N}}$ be a matrix indexed by $\mathbb{N}\times \mathbb{N}$, where each entry $m_{ij}$ is a meromorphic function on $\mathbb{C}^r$. Further, let $u=(u_j)$ and $v=(v_j)$ be column vectors indexed by $\mathbb{N}$, whose entries are meromorphic functions on $\mathbb{C}^r$. Then we say that $Mu=v$ if, for every $i\in \mathbb{N}$, the series
$\sum_{j\in \mathbb{N}} m_{ij}u_j$ converges normally on any compact subset of $\mathbb{C}^r$ and the resulting sum equals $v_j$. Using these notations, the translation formula \eqref{eq14} can be regarded as the first element of an infinite family of identities arising from successive applications of the shift $s_1\r s_1+1$  to both sides of \eqref{eq14}. Adopting this viewpoint, we formulate the entire family of relations in terms of infinite matrices, as follows: 
\begin{align}
	\label{eq19}
	N(s_1 )V(s_1+s_2, s_3, \dots , s_r)=M(s_1)V(s_1, s_2, \dots, s_r),
\end{align}
where 
\begin{align}
	\label{equn5.16}
	V(s_1, s_2, \dots, s_r)=	
	\begin{pmatrix}
		\z_q(s_1, s_2, \dots , s_r)\\
		\z_q(s_1+1, s_2, \dots , s_r)\\
		\z_q(s_1+2, s_2, \dots , s_r)\\
		\vdots 
	\end{pmatrix},
\end{align}
\begin{align}
	\label{equn5.17}
	M(t)=	
	\begin{pmatrix}
		1 & \dfrac{t}{(q^{-t}-1)} & -\dfrac{t(t+1)}{2!(q^{-t}-1)} & \dfrac{t(t+1)(t+2)}{3!(q^{-t}-1)} & -\dfrac{t(t+1)(t+2)(t+3)}{4!(q^{-t}-1)} & \dots \\
		0 & 1 &\dfrac{t+1}{(q^{-(t+1)}-1)} & -\dfrac{(t+1)(t+2)}{2!(q^{-(t+1)}-1)} & \dfrac{(t+1)(t+2)(t+3)}{3!(q^{-(t+1)}-1)} & \dots \\
		0 & 0 & 1 & \dfrac{t+2}{(q^{-(t+2)}-1)} & -\dfrac{(t+2)(t+3)}{2!(q^{-(t+2)}-1)} & \dots \\
		\vdots & \vdots & \vdots & \vdots & \vdots & \ddots
	\end{pmatrix}
\end{align}
\begin{align*}
	\text{and }~ N(t)=
	\begin{pmatrix}
		\dfrac{1}{(q^{-t}-1)} & -\dfrac{t}{(q^{-t}-1)} & \dfrac{t(t+1)}{2!(q^{-t}-1)} & -\dfrac{t(t+1)(t+2)}{3!(q^{-t}-1)}  & \dots \\
		0 & \dfrac{1}{(q^{-(t+1)}-1)} &-\dfrac{t+1}{(q^{-(t+1)}-1)} & +\dfrac{(t+1)(t+2)}{2!(q^{-(t+1)}-1)}  & \dots \\
		0 & 0 & \dfrac{1}{(q^{-(t+2)}-1)} & -\dfrac{t+2}{(q^{-(t+2)}-1)} &  \dots \\
		0 &0 &0 & \dfrac{1}{(q^{-(t+3)}-1)} & \dots\\ 
		\vdots & \vdots & \vdots & \vdots  & \ddots 
	\end{pmatrix}.
\end{align*}
Our translation formula \eqref{eq14} is given by the first identity obtained from \eqref{eq19}. 
Since $M(s_1)$ is an upper triangular matrix with all the diagonal entries $1$, $M(s_1)^{-1}$ exists. We will show in due time that one can express the column vector $V(s_1, \dots, s_r)$ in terms of $V(s_1+s_2, s_3, \dots, s_r)$ by multiplying both sides of \eqref{eq19} by $M(s_1)^{-1}$. First we need to obtain $M(s_1)^{-1}$. 
By \cite[Theorem 3]{9}, we get
\begin{align*}
	M(t)^{-1}={(m_{i,j})}_{i,j \in \mathbb{N}}, 
\end{align*}
where  
$$m_{k,n} =
\begin{cases}
	1& \text{if $k=n$},\\
	{(-1)}^{n-k }D_{k,n}(t)& \text{if $k \leq n-1$},\\
	0& \text{if $k>n$}.
\end{cases}
$$
Here for $2 \leq k \leq n-1$, $D_{k,n}(t)$ is given by the recurrence relation $D_{k,n}(t)=D_{1,(n-k+1)}(t+k-1)$, where 
\begin{align*}
	D_{1,n}(t)	=
	\begin{vmatrix}
		\dfrac{t}{(q^{-t}-1)} & -\dfrac{t(t+1)}{2!(q^{-t}-1)} & \dfrac{t(t+1)(t+2)}{3!(q^{-t}-1)} &  \dots &{(-1)}^{n}\dfrac{t(t+1) \cdots (t+n-2)}{(n-1)!(q^{-t}-1) } \\
		1 &\dfrac{t+1}{(q^{-(t+1)}-1)} & -\dfrac{(t+1)(t+2)}{2!(q^{-(t+1)}-1)}& \dots & {(-1)}^{n-1}\dfrac{(t+1)(t+2)\cdots (t+n-2)}{(n-2)!(q^{-(t+1)}-1) } \\
		0 & 1 & \dfrac{t+2}{(q^{-(t+2)}-1)} & \dots & {(-1)}^{n-2}\dfrac{(t+2)(t+3)\cdots (t+n-2)}{(n-3)!(q^{-(t+2)}-1) } \\
		\vdots & \vdots & \vdots & \vdots & \vdots \\
		0 & 0 &0 & \dots 1 & \dfrac{(t+n-2)}{(q^{-(t+n-2)}-1)} 
	\end{vmatrix}.
\end{align*}
In this particular case, we will find the complete expansion of the determinant $D_{1, n}(t)$, hence obtaining the expansion of $D_{k, n}(t)$. We will then use $D_{1, n}(t)$ to obtain the inverse of the infinite triangular matrix $M(t)$. At this point, we prove the following proposition.
\begin{proposition}
	\label{p5.3}
	The value of the determinant
	\begin{align*}
		\mathcal{G}=
		\begin{vmatrix}
			a_{1,1} & -a_{1,2}& a_{1,3} &  -a_{1,4} & \dots &{(-1)}^{n-1}a_{1,n} \\
			a_{2,1} &a_{2,2} & -a_{2,3}& a_{2,4} &  \dots &{(-1)}^{n-2}a_{2,n} \\
			0 & a_{3,2} & a_{3,3} & -a_{3,4} &  \dots &{(-1)}^{n-3}a_{3,n}\\
			\vdots & \vdots & \vdots & \vdots & \vdots & \vdots \\
			0 & 0 &0 & 0 &\dots a_{n,n-1} & a_{n, n} 
		\end{vmatrix}
	\end{align*}
	is given by \begin{align}
		\label{eq6.222}
		\sum\limits_{\substack{\sigma \in \mathfrak{S}_n; \sigma(j) \geq j-1,\\ 2 \leq j \leq n}} a_{1,\sigma(1)}a_{2,\sigma(2)} \dots a_{n,\sigma(n)},
	\end{align} 
	where $\mathfrak{S}_n$ is the set of all permutations on $\{1, 2, \dots, n\}$.
\end{proposition}
\begin{proof}
	By the recurrence relation of the determinant of an upper Hessenberg matrix \cite{14}
	\begin{align*}
		\mathcal{C}=
		\begin{pmatrix}
			a_{1,1} & a_{1,2}& a_{1,3} &  a_{1,4} & \dots &a_{1,n} \\
			a_{2,1} &a_{2,2} & a_{2,3}& a_{2,4} &  \dots &a_{2,n} \\
			0 & a_{3,2} & a_{3,3} & a_{3,4} &  \dots &a_{3,n}\\
			\vdots & \vdots & \vdots & \vdots & \vdots & \vdots \\
			0 & 0 &0 & 0 &\dots a_{n,n-1} & a_{n, n} 
		\end{pmatrix},
	\end{align*} we have
	\begin{align}
		\label{e5.22}
		d_0=1, d_k=a_{k,k}d_{k-1}+\sum_{i=1}^{k-1}\big[{(-1)}^{k-i}a_{i,k}\prod_{j=i}^{k-1}a_{j+1,j}d_{i-1}\big],
	\end{align} where $d_k$ is the determinant of the top-left $k\times k$ submatrix. Then $d_n$ is the determinant value of $\mathcal{C}$.
	Comparing $\mathcal{C}$ with the matrix given in $\mathcal{G}$, we can see that the entries $a_{i,k}$ have negative sign when $k>i$ and $k-i $ is odd. So in  case of $\mathcal{G}$, \eqref{e5.22} gives
	\begin{align*}
		d_0=1, d_k=a_{k,k}d_{k-1}+\sum_{i=1}^{k-1}\big[a_{i,k}\prod_{j=i}^{k-1}a_{j+1,j}d_{i-1}\big].
	\end{align*} 
	Therefore, $d_1=a_{1,1}$, $d_2=a_{2,2}a_{1,1}+a_{1,2}a_{2,1}$.
	We prove \eqref{eq6.222} by induction on $k$. Suppose it is true for $d_j,$ where $ j \leq k-1$.
	By the induction hypothesis,
	\begin{align*}
		d_k&=a_{k,k}\sum\limits_{\substack{\sigma \in \mathfrak{S}_{k-1}; \sigma(j) \geq j-1,\\ 2 \leq j \leq k-1}} a_{1,\sigma(1)} \dots a_{k-1,\sigma(k-1)} + a_{1,k}\prod_{j=1}^{k-1}a_{j+1,j}\\
		&+\sum_{i=2}^{k-1}\Bigg[a_{i,k}\prod_{j=i}^{k-1}a_{j+1,j}\sum\limits_{\substack{\sigma \in \mathfrak{S}_{i-1}; \sigma(j) \geq j-1,\\ 2 \leq j \leq i-1}} a_{1,\sigma(1)} \dots a_{i-1,\sigma(i-1)}\Bigg]\\
		&=\sum\limits_{\substack{\sigma \in \mathfrak{S}_k; \sigma(j) \geq j-1,\\ 2 \leq j \leq k}} a_{1,\sigma(1)}a_{2,\sigma(2)} \dots a_{k,\sigma(k)}.	
	\end{align*}
	Hence proved.
\end{proof}
Now we give an application of Proposition \ref{p5.3}.
\begin{example}
	Let $q_i=q_i(t)=q^{-(t+i)}-1$. Then 
	\begin{align}
		\label{equ5.177}
		&D_{1,6}(t)\nonumber\\
		&=
		\begin{vmatrix}
			\dfrac{t}{(q^{-t}-1)} & -\dfrac{t(t+1)}{2!(q^{-t}-1)} & \dfrac{t(t+1)(t+2)}{3!(q^{-t}-1)} &-\dfrac{t(t+1)(t+2)(t+3)}{4!(q^{-t}-1)} & \dfrac{t(t+1) \cdots (t+4)}{5!(q^{-t}-1) } \\
			1 &\dfrac{t+1}{(q^{-(t+1)}-1)} & -\dfrac{(t+1)(t+2)}{2!(q^{-(t+1)}-1)}& \dfrac{(t+1)(t+2)(t+3)}{3!(q^{-(t+1)}-1)} & -\dfrac{(t+1)\cdots (t+4)}{4!(q^{-(t+1)}-1) } \\
			0 & 1 & \dfrac{t+2}{(q^{-(t+2)}-1)} & -\dfrac{(t+2)(t+3)}{2!(q^{-(t+2)}-1)} & \dfrac{(t+2)(t+3) (t+4)}{3!(q^{-(t+2)}-1) }  \\
			0 & 0 & 1 & \dfrac{t+3}{(q^{-(t+3)}-1)} &  -\dfrac{(t+3)(t+4)}{2!(q^{-(t+3)}-1)}\\
			0 & 0 &0 &  1 & \dfrac{(t+4)}{(q^{-(t+4)}-1)} 
		\end{vmatrix}
		\nonumber\\
		&=t(t+1)(t+2)(t+3)(t+4)\Bigg\{ \dfrac{1}{q_0q_1q_2q_3q_4}+\dfrac{1}{2!}\bigg(\dfrac{1}{q_0q_2q_3q_4}+ \dfrac{1}{q_0q_1q_3q_4}+ \dfrac{1}{q_0q_1q_2q_4}+\dfrac{1}{q_0q_1q_2q_3}\bigg)\nonumber\\
		&  +\dfrac{1}{3!}\bigg(\dfrac{1}{q_0q_3q_4}+\dfrac{1}{q_0q_1q_4}+\dfrac{1}{q_0q_1q_2}\bigg)+\dfrac{1}{4!}\bigg(\dfrac{1}{q_0q_4}+\dfrac{1}{q_0q_1}\bigg)+\dfrac{1}{2!2!}\bigg(\dfrac{1}{q_0q_2q_4}+\dfrac{1}{q_0q_2q_3}+\dfrac{1}{q_0q_1q_3}\bigg)\nonumber\\
		&  + \dfrac{1}{5!}\dfrac{1}{q_0} + \dfrac{1}{2!3!}\bigg(\dfrac{1}{q_0q_2}+\dfrac{1}{q_0q_3}\bigg)\Bigg\}.
	\end{align}
\end{example}
\begin{proof}
	Comparing $D_{1, 6}(t)$ with $\mathcal{G}$ for $n=5$ in Proposition \ref{p5.3}, \eqref{eq6.222} gives
	\begin{align*}
		D_{1, 6}(t)=&a_{11}a_{22}a_{33}a_{44}a_{55}+a_{11}a_{22}a_{33}a_{45}a_{54}+a_{11}a_{22}a_{34}a_{43}a_{55}+a_{11}a_{22}a_{35}a_{43}a_{54}+a_{11}a_{23}a_{32}a_{44}a_{55}\\
		& +a_{11}a_{23}a_{32}a_{45}a_{54} +a_{11}a_{24}a_{32}a_{43}a_{55}+a_{11}a_{25}a_{32}a_{43}a_{54}+a_{12}a_{21}a_{33}a_{44}a_{55}+a_{12}a_{21}a_{33}a_{45}a_{54}\\
		& +a_{12}a_{21}a_{34}a_{43}a_{55}+a_{12}a_{21}a_{35}a_{43}a_{54}+a_{13}a_{21}a_{32}a_{44}a_{55}+a_{13}a_{21}a_{32}a_{45}a_{54}+a_{14}a_{21}a_{32}a_{43}a_{55}\\
		& +a_{15}a_{21}a_{32}a_{43}a_{54}.
	\end{align*}
	Since $a_{j+1, j}=1$, for $j=1, \dots 4$, 
	\begin{align*}
		D_{1,6}(t)=&a_{11}a_{22}a_{33}a_{44}a_{55}+a_{11}a_{22}a_{33}a_{45}+a_{11}a_{22}a_{34}a_{55}+a_{11}a_{22}a_{35}+a_{11}a_{23}a_{44}a_{55}\\
		& +a_{11}a_{23}a_{45} +a_{11}a_{24}a_{55}+a_{11}a_{25}+a_{12}a_{33}a_{44}a_{55}+a_{12}a_{33}a_{45}\\
		& +a_{12}a_{34}a_{55}+a_{12}a_{35}+a_{13}a_{44}a_{55}+a_{13}a_{45}+a_{14}a_{55} +a_{15}.
	\end{align*}
	Each term of $D_{1,6}(t)$ contains some factorial quantity. We collect the terms containing the same factorial and keep them together. The only term containing $\dfrac{1}{1!}$ is 
	$$a_{11}a_{22}a_{33}a_{44}a_{55}=t(t+1)(t+2)(t+3)(t+4) \dfrac{1}{q_0q_1q_2q_3q_4}.$$
	The terms containing $$ \dfrac{1}{2!}, \dfrac{1}{3!}, \dfrac{1}{4!}, \dfrac{1}{5!}, \dfrac{1}{2!2!} \text{ and } \dfrac{1}{3!2!},$$ are respectively given by the sets 
	$\{a_{12}a_{21}a_{33}a_{44}a_{55}, a_{11}a_{23}a_{44}a_{55}, a_{11}a_{22}a_{34}a_{55}, a_{11}a_{22}a_{33}a_{45}\}$, \\ $\{a_{13}a_{44}a_{55}, a_{11}a_{24}a_{55}, a_{11}a_{22}a_{35}\} $, $\{a_{14}a_{55}, a_{11}a_{25}\}$, $\{a_{15}\}$, $\{a_{12}a_{33}a_{45}, a_{12}a_{34}a_{55}, a_{11}a_{23}a_{45}$\}, and\\ \{$a_{12}a_{35}, a_{13}a_{45}\}$.
	Adding the terms within each class and then summing the resulting expressions, we get \eqref{equ5.177}.

\end{proof}
The observations made in the last example lead to the following result:
\begin{proposition}
	\label{p5.4}
	Let $n \geq 2$. The value of the determinant $D_{1,n}(t)$ is given by 
	\begin{align*}
		t(t+1) \cdots (t+n-2)\mathcal{L}_n(t),
	\end{align*}
	where 
	\begin{align}
		\label{eq22}
		&\mathcal{L}_n(t)\nonumber\\
		&=\prod_{i=0}^{n-2}\dfrac{1}{ q_i}+\sum_{i=2}^{n-1}\sum_{\substack{(i_1,  \dots , i_j) \in P(i),\\ 1 \leq j \leq [\frac{i}{2}]}} \dfrac{1}{i_1!i_2! \dots i_j!}\sum_{0 \leq k_j \leq k_{j-1} \leq \dots \leq k_1 \leq n-i-1}\prod\limits_{u=0}^{k_j}\dfrac{1}{q_u}\prod\limits_{r=1}^{j-1}\prod_{s=k_{j-r+1}}^{k_{j-r}}\dfrac{1}{q_{(T_r+s)}} \prod\limits_{v=i+k_1}^{n-2}\dfrac{1}{q_v},
	\end{align}
	$P(i)=\{(i_1, i_2, \dots ,i_j):i_1, \dots ,i_j \in \mathbb{N}\backslash \{1\}, j \geq 1, i_1+i_2+ \dots +i_j=i\}$ , $q_i=q_i(t)=q^{-(t+i)}-1$ and
	$T_r=i_1+i_2+ \dots +i_r.
	$
	We assume that $\prod_{v=v_1}^{v_2}\frac{1}{q_v}=1$, when  $v_1>v_2.$
\end{proposition}
\begin{proof}
	We use induction on $n$.
	Let $n=2$. Then 
	\begin{align*}
		D_{1,2}(t)=
		\dfrac{t}{(q^{-t}-1)}=t\mathcal{L}_2(t).
	\end{align*}
	Therefore, for $n=2$, the result is true.
	Suppose that $n >2$ and the statement of the proposition is true for all determinants $D_{1,l}(t)$ of order $(l-1)\times(l-1),$ where $l \leq n.$  
	For an $n \times n$ upper Hessenberg matrix with entries $a_{i,k}$, the following recurrence relation holds \cite{14}
	\begin{align}
		\label{e5.223}
		d_0=1, d_k=a_{k,k}d_{k-1}+\sum_{i=1}^{k-1}\big[{(-1)}^{k-i}a_{i,k}\prod_{j=i}^{k-1}a_{j+1,j}d_{i-1}\big],
	\end{align}
	where  $d_k$ is the determinant of the top-left $k\times k$ submatrix. Then $d_n$ is the determinant value.
	In our case, $D_{1,n}(t)$ is the determinant of an upper Hessenberg matrix with the entries $a_{i,k}$ having a negative sign when $k>i$ and $k-i $ is odd. Moreover, $a_{j+1,j}=1$ for $j \geq 1$. So in this case, \eqref{e5.223} gives
	\begin{align*}
		d_0=1,~ d_k=a_{k,k}d_{k-1}+\sum_{i=1}^{k-1}\big[a_{i,k}d_{i-1}\big].
	\end{align*} 
	Therefore the value of the determinant $D_{1,n+1}(t)$ of order $n \times n$ is given by
	\begin{align*}
		D_{1,n+1}(t) = \sum_{l=1}^{n}\dfrac{(t+l-1)(t+l) \cdots(t+n-1) }{(n-l+1)!q_{l-1}}D_{1,l}(t).  
	\end{align*}
	By the induction hypothesis, we have
	\begin{align*}
		&D_{1,n+1}(t)\nonumber\\
		&=\dfrac{t(t+1) \cdots (t+n-1)}{n!q_0}+\dfrac{(t+1) \cdots (t+n-1)}{(n-1)!q_1}\cdot \dfrac{t}{q_0}\nonumber \\
		& \ \ \ +\sum_{l=3}^{n}\dfrac{(t+l-1)(t+l) \cdots(t+n-1) }{(n-l+1)!q_{l-1}}\cdot t(t+1) \cdots (t+l-2) \Bigg[\prod_{i=0}^{l-2}\dfrac{1}{ q_i}\nonumber \\
		& \ \ +\sum_{i=2}^{l-1}\sum_{\substack{(i_1, i_2, \dots , i_j) \in P(i), j \geq 1}} \dfrac{1}{i_1!i_2! \dots i_j!}\sum_{0 \leq k_j \leq k_{j-1} \leq \dots \leq k_1 \leq l-i-1}\prod\limits_{u=0}^{k_j}\dfrac{1}{q_u}\prod\limits_{r=1}^{j-1}\prod_{s=k_{j-r+1}}^{k_{j-r}}\dfrac{1}{q_{(T_r+s)}}\prod\limits_{v=i+k_1}^{l-2}\dfrac{1}{q_v} \Bigg]\\
		&= t(t+1) \cdots (t+n-1)\Bigg[\dfrac{1}{n!q_0}+\sum_{l=2}^{n}\dfrac{1}{(n-l+1)!q_{l-1}}\prod_{i=0}^{l-2}\dfrac{1}{ q_i}+\sum_{l=3}^{n}\dfrac{1}{(n-l+1)!q_{l-1}}\times  \nonumber \\
		& \ \ \sum_{i=2}^{l-1}\sum_{\substack{(i_1, i_2, \dots , i_j) \in P(i), j \geq 1}} \dfrac{1}{i_1!i_2! \dots i_j!}\sum_{0 \leq k_j \leq k_{j-1} \leq \dots \leq k_1 \leq l-i-1}\prod\limits_{u=0}^{k_j}\dfrac{1}{q_u}\prod\limits_{r=1}^{j-1}\prod_{s=k_{j-r+1}}^{k_{j-r}}\dfrac{1}{q_{(T_r+s)}}\prod\limits_{v=i+k_1}^{l-2}\dfrac{1}{q_v} \Bigg].\\
	\end{align*}
	For fixed positive integers $i, a$ and $n$, let us define
	\begin{align*}
		&\mathfrak{U}_{i,a}^{n}(t) =\sum_{\substack{(i_1, i_2, \dots , i_j) \in P(i),\\ a\leq j \leq [\frac{i}{2}]}} \dfrac{1}{i_1!i_2! \dots i_j!}\sum_{0 \leq k_j \leq k_{j-1} \leq \dots \leq k_1 \leq n-i-1}\prod\limits_{u=0}^{k_j}\dfrac{1}{q_u}\prod\limits_{r=1}^{j-1}\prod_{s=k_{j-r+1}}^{k_{j-r}}\dfrac{1}{q_{(T_r+s)}}\prod\limits_{v=i+k_1}^{n-2}\dfrac{1}{q_v}.  
	\end{align*}
	Using this notation, we have
	\begin{align*}
		\mathfrak{U}_{i,1}^{n}(t)= \frac{1}{i!}\sum_{0  \leq k_1 \leq n-i-1}\prod\limits_{u=0}^{k_1}\dfrac{1}{q_u}\prod\limits_{v=i+k_1}^{n-2}\dfrac{1}{q_v}+\mathfrak{U}_{i,2}^{n}(t).
	\end{align*}
	Also
	\begin{align*}
		\sum_{i=2}^{n-1}\mathfrak{U}_{i,1}^{n}(t)=\sum_{i=2}^{n-1} \frac{1}{i!}\sum_{0  \leq k_1 \leq n-i-1}\prod\limits_{u=0}^{k_1}\dfrac{1}{q_u}\prod\limits_{v=i+k_1}^{n-2}\dfrac{1}{q_v}+\sum_{i=4}^{n-1}\mathfrak{U}_{i,2}^{n}(t).
	\end{align*}
	Then 
	\begin{align*}
		&D_{1,n+1}(t)\\
		&=t(t+1) \cdots (t+n-1)\Bigg[\sum_{l=1}^{n}\dfrac{1}{(n-l+1)!}\prod_{i=0}^{l-1}\dfrac{1}{ q_i}+\sum_{i=2}^{n-1}\frac{1}{q_{n-1}}\mathfrak{U}_{i,1}^{n}(t) + \sum_{l=3}^{n-1}\dfrac{1 }{(n-l+1)!q_{l-1}}\sum_{i=2}^{l-1}\mathfrak{U}_{i,1}^{l}(t)\Bigg]\\
		&=t(t+1) \cdots (t+n-1)\Bigg[\prod_{i=0}^{n-1}\dfrac{1}{ q_i}+\sum_{l=1}^{n-1}\dfrac{1}{(n-l+1)!}\prod_{i=0}^{l-1}\dfrac{1}{ q_i}+\sum_{i=2}^{n-1}\frac{1}{i!}\sum_{0  \leq k_1 \leq n-i-1}\prod\limits_{u=0}^{k_1}\dfrac{1}{q_u}\prod\limits_{v=i+k_1}^{n-1}\dfrac{1}{q_v}\\
		& \ \ +\sum_{i=4}^{n-1}\frac{1}{q_{n-1}}\mathfrak{U}_{i,2}^{n}(t)+\sum_{l=3}^{n-1}\dfrac{1 }{(n-l+1)!q_{l-1}}\sum_{i=2}^{l-1}\mathfrak{U}_{i,1}^{l}(t)\Bigg].
	\end{align*}
	Combining the second and third terms except for $l=1$, the R. H. S. of the above equation is equal to
	\begin{align}
		\label{eq5.301}
		&t(t+1) \cdots (t+n-1)\Bigg[ \prod_{i=0}^{n-1}\dfrac{1}{ q_i}+ \dfrac{1}{n!q_0}+\sum_{i=2}^{n-1}\frac{1}{i!}\sum_{0  \leq k_1 \leq n-i}\prod\limits_{u=0}^{k_1}\dfrac{1}{q_u}\prod\limits_{v=i+k_1}^{n-1}\dfrac{1}{q_v} +\sum_{i=4}^{n-1}\frac{1}{q_{n-1}}\mathfrak{U}_{i,2}^{n}(t)\nonumber\\
		& \ \ +\sum_{l=3}^{n-1}\dfrac{1 }{(n-l+1)!q_{l-1}}\sum_{i=2}^{l-1}\mathfrak{U}_{i,1}^{l}(t)\Bigg] \nonumber\\
		&=t(t+1) \cdots (t+n-1)\Bigg[ \prod_{i=0}^{n-1}\dfrac{1}{ q_i}+ \sum_{i=2}^{n}\frac{1}{i!}\sum_{0  \leq k_1 \leq n-i}\prod\limits_{u=0}^{k_1}\dfrac{1}{q_u}\prod\limits_{v=i+k_1}^{n-1}\dfrac{1}{q_v} +\sum_{i=4}^{n-1}\frac{1}{q_{n-1}}\mathfrak{U}_{i,2}^{n}(t)\nonumber\\
		& \ \ +\sum_{l=3}^{n-1}\dfrac{1 }{(n-l+1)!q_{l-1}}\sum_{i=2}^{l-1}\mathfrak{U}_{i,1}^{l}(t)\Bigg].
	\end{align}
	Now 
	\begin{align}
		\label{eq5.311}
		&\sum_{l=3}^{n-1}\dfrac{1 }{(n-l+1)!q_{l-1}}\sum_{i=2}^{l-1}\mathfrak{U}_{i,1}^{l}(t)\nonumber \\
		&=\dfrac{1 }{(n-2)!q_{2}}\sum_{i=2}^{2}\mathfrak{U}_{i,1}^{3}(t) + \dfrac{1 }{(n-3)!q_{3}}\sum_{i=2}^{3}\mathfrak{U}_{i,1}^{4}(t)+ \dots + \dfrac{1 }{2!q_{n-2}}\sum_{i=2}^{n-2}\mathfrak{U}_{i,1}^{n-1}(t) \nonumber\\
		&=\frac{1}{2!q_{n-2}}\mathfrak{U}_{2,1}^{n-1}(t)+\bigg\{\frac{1}{2!q_{n-2}}\mathfrak{U}_{3,1}^{n-1}(t)+\frac{1}{3!q_{n-3}}\mathfrak{U}_{2,1}^{n-2}(t)\bigg\}+\bigg\{\frac{1}{2!q_{n-2}}\mathfrak{U}_{4,1}^{n-1}(t)+\frac{1}{3!q_{n-3}}\mathfrak{U}_{3,1}^{n-2}(t)+\nonumber \\
		& \ \ \ \frac{1}{4!q_{n-4}}\mathfrak{U}_{2,1}^{n-3}(t)\bigg\} + \dots +\bigg\{\frac{1}{2!q_{n-2}}\mathfrak{U}_{n-2,1}^{n-1}(t)+\frac{1}{3!q_{n-3}}\mathfrak{U}_{n-3,1}^{n-2}(t)+\dots + \frac{1}{(n-2)!q_{2}}\mathfrak{U}_{2,1}^{3}(t)\bigg\}\nonumber \\
		&=\frac{1}{2!q_{n-2}}\mathfrak{U}_{2,1}^{n-1}(t)+ \sum_{\alpha=2}^{3}\frac{1}{\alpha!q_{n-\alpha}}\mathfrak{U}_{5-\alpha,1}^{n-\alpha+1}(t)+\sum_{\alpha=2}^{4}\frac{1}{\alpha!q_{n-\alpha}}\mathfrak{U}_{6-\alpha,1}^{n-\alpha+1}(t)+ \dots +\sum_{\alpha=2}^{n-2}\frac{1}{\alpha!q_{n-\alpha}}\mathfrak{U}_{n-\alpha,1}^{n-\alpha+1}(t)\nonumber \\
		&=\sum_{i=4}^{n}\sum_{\alpha=2}^{i-2}\frac{1}{\alpha!q_{n-\alpha}}\mathfrak{U}_{i-\alpha,1}^{n-\alpha+1}(t).
	\end{align}
	Using \eqref{eq5.311} in \eqref{eq5.301}, we get
	\begin{align}
		\label{e5.332}
		&D_{1,n+1}(t)\nonumber \\
		&=t(t+1) \cdots (t+n-1)\Bigg[ \prod_{i=0}^{n-1}\dfrac{1}{ q_i}+ \sum_{i=2}^{n}\frac{1}{i!}\sum_{0  \leq k_1 \leq n-i}\prod\limits_{u=0}^{k_1}\dfrac{1}{q_u}\prod\limits_{v=i+k_1}^{n-1}\dfrac{1}{q_v} +\sum_{i=4}^{n-1}\frac{1}{q_{n-1}}\mathfrak{U}_{i,2}^{n}(t)\nonumber\\
		& \ \ +\sum_{i=4}^{n}\sum_{\alpha=2}^{i-2}\frac{1}{\alpha!}\mathfrak{U}_{i-\alpha,1}^{n-\alpha+1}(t)\Bigg].
	\end{align}
	Next we will show that
	\begin{align}
		\label{e6.331}
		\sum_{i=4}^{n-1}\frac{1}{q_{n-1}}\mathfrak{U}_{i,2}^{n}(t)+\sum_{i=4}^{n}\sum_{\alpha=2}^{i-2}\frac{1}{\alpha!}\mathfrak{U}_{i-\alpha,1}^{n-\alpha+1}(t)  = \sum_{i=4}^{n}\mathfrak{U}_{i,2}^{n+1}(t).
	\end{align}
	We have
	\begin{align*}
		\mathfrak{U}_{i,2}^{n+1}(t)&=\sum_{\substack{(i_1, i_2, \dots , i_j) \in P(i),\\ 2\leq j \leq [\frac{i}{2}]}} \dfrac{1}{i_1!i_2! \dots i_j!}\sum_{0 \leq k_j \leq k_{j-1} \leq \dots \leq k_1 \leq n-i}\prod\limits_{u=0}^{k_j}\dfrac{1}{q_u}\prod\limits_{r=1}^{j-1}\prod_{s=k_{j-r+1}}^{k_{j-r}}\dfrac{1}{q_{(T_r+s)}}\prod\limits_{v=i+k_1}^{n-1}\dfrac{1}{q_v}\nonumber \\
		&=\sum_{\substack{(i_1, i_2, \dots , i_j) \in P(i),\\ 2\leq j \leq [\frac{i}{2}]}} \dfrac{1}{i_1!i_2! \dots i_j!}\sum_{0 \leq k_j \leq k_{j-1} \leq \dots \leq k_1 \leq n-i-1}\prod\limits_{u=0}^{k_j}\dfrac{1}{q_u}\prod\limits_{r=1}^{j-1}\prod_{s=k_{j-r+1}}^{k_{j-r}}\dfrac{1}{q_{(T_r+s)}}\prod\limits_{v=i+k_1}^{n-1}\dfrac{1}{q_v}\nonumber \\
		& \ \ +\sum_{\substack{(i_1, i_2, \dots , i_j) \in P(i),\\ 2\leq j \leq [\frac{i}{2}]}} \dfrac{1}{i_1!i_2! \dots i_j!}\sum_{0 \leq k_j \leq k_{j-1} \leq \dots \leq k_2 \leq n-i}\prod\limits_{u=0}^{k_j}\dfrac{1}{q_u}\prod\limits_{r=1}^{j-2}\prod_{s=k_{j-r+1}}^{k_{j-r}}\dfrac{1}{q_{(T_r+s)}}\prod_{s=k_2}^{n-i}\dfrac{1}{q_{(T_{j-1}+s)}}. 
	\end{align*}
	Taking the summation over $i$ from 4 to $n$ on both sides of the above equation, we get
	\begin{align*}
		\sum_{i=4}^{n}\mathfrak{U}_{i,2}^{n+1}(t)
		&=\sum_{i=4}^{n-1}\frac{1}{q_{n-1}}\mathfrak{U}_{i,2}^{n}(t)\nonumber \\ & \ \ +\sum_{i=4}^{n}\sum_{\substack{(i_1, i_2, \dots , i_j) \in P(i),\\ 2\leq j \leq [\frac{i}{2}]}} \dfrac{1}{i_1!i_2! \dots i_j!}\sum_{0 \leq k_j \leq k_{j-1} \leq \dots \leq k_2 \leq n-i}\prod\limits_{u=0}^{k_j}\dfrac{1}{q_u}\prod\limits_{r=1}^{j-2}\prod_{s=k_{j-r+1}}^{k_{j-r}}\dfrac{1}{q_{(T_r+s)}}\prod_{s=k_2}^{n-i}\dfrac{1}{q_{(T_{j-1}+s)}}.
	\end{align*}
	To prove \eqref{e6.331}, we claim that
	\begin{align*}
		&\sum_{\alpha=2}^{i-2}\frac{1}{\alpha!q_{n-\alpha}}\mathfrak{U}_{i-\alpha,1}^{n-\alpha+1}(t) \\
		&= \sum_{\substack{(i_1, i_2, \dots , i_j) \in P(i),\\ 2\leq j \leq [\frac{i}{2}]}} \dfrac{1}{i_1!i_2! \dots i_j!}\sum_{0 \leq k_j \leq k_{j-1} \leq \dots \leq k_2 \leq n-i}\prod\limits_{u=0}^{k_j}\dfrac{1}{q_u}\prod\limits_{r=1}^{j-2}\prod_{s=k_{j-r+1}}^{k_{j-r}}\dfrac{1}{q_{(T_r+s)}}\prod_{s=k_2}^{n-i}\dfrac{1}{q_{(T_{j-1}+s)}}.
	\end{align*}
	We can see
	\begin{align}
		\label{eq6.341}
		&\sum_{\alpha=2}^{i-2}\frac{1}{\alpha!q_{n-\alpha}}\mathfrak{U}_{i-\alpha,1}^{n-\alpha+1}(t)\nonumber \\
		&=\sum_{\alpha=2}^{i-2}\frac{1}{\alpha!}\sum_{\substack{(i_1, i_2, \dots , i_j) \in P(i-\alpha),\\ 1\leq j \leq [\frac{i-\alpha}{2}]}} \dfrac{1}{i_1!i_2! \dots i_j!}\sum_{0 \leq k_j \leq k_{j-1} \leq \dots \leq k_1 \leq n-i}\prod\limits_{u=0}^{k_j}\dfrac{1}{q_u}\prod\limits_{r=1}^{j-1}\prod_{s=k_{j-r+1}}^{k_{j-r}}\dfrac{1}{q_{(T_r+s)}}\prod_{v=i-\alpha+k_1}^{n-\alpha}\dfrac{1}{q_{v}}\nonumber \\
		&=\sum_{j=1}^{[\frac{i}{2}]-1}\sum_{\alpha=2}^{i-2j}\frac{1}{\alpha!}\sum_{\substack{(i_1, i_2, \dots , i_j) \in P(i-\alpha)}} \dfrac{1}{i_1!i_2! \dots i_j!}\sum_{0 \leq k_j \leq k_{j-1} \leq \dots \leq k_1 \leq n-i}\prod\limits_{u=0}^{k_j}\dfrac{1}{q_u}\prod\limits_{r=1}^{j-1}\prod_{s=k_{j-r+1}}^{k_{j-r}}\dfrac{1}{q_{(T_r+s)}}\prod_{v=i-\alpha+k_1}^{n-\alpha}\dfrac{1}{q_{v}} 
	\end{align}
	and 
	\begin{align*}
		& \sum_{\substack{(i_1, i_2, \dots , i_j) \in P(i),\\ 2\leq j \leq [\frac{i}{2}]}} \dfrac{1}{i_1!i_2! \dots i_j!}\sum_{0 \leq k_j \leq k_{j-1} \leq \dots \leq k_2 \leq n-i}\prod\limits_{u=0}^{k_j}\dfrac{1}{q_u}\prod\limits_{r=1}^{j-2}\prod_{s=k_{j-r+1}}^{k_{j-r}}\dfrac{1}{q_{(T_r+s)}}\prod_{s=k_2}^{n-i}\dfrac{1}{q_{(T_{j-1}+s)}}\nonumber \\
		&=\sum_{j=1}^{[\frac{i}{2}]-1} \sum_{\substack{(i_1, i_2, \dots , i_{j+1}) \in P(i)}} \dfrac{1}{i_1!i_2! \dots i_{j+1}!}\sum_{0 \leq k_{j+1} \leq k_{j} \leq \dots \leq k_2 \leq n-i}\prod\limits_{u=0}^{k_{j+1}}\dfrac{1}{q_u}\prod\limits_{r=1}^{j-1}\prod_{s=k_{j-r+2}}^{k_{j-r+1}}\dfrac{1}{q_{(T_r+s)}}\prod_{s=k_2}^{n-i}\dfrac{1}{q_{(T_{j}+s)}}.
	\end{align*}
	Let $i_{j+1}=\alpha$ in the above expression. Depending on $i$ and $ j$, $\alpha$ can take the values 2, 3, \dots , $i-2j$. Therefore putting these values of $\alpha$ in the above equation, we get
	\begin{align*}
		& \sum_{\substack{(i_1, i_2, \dots , i_j) \in P(i),\\ 2\leq j \leq [\frac{i}{2}]}} \dfrac{1}{i_1!i_2! \dots i_j!}\sum_{0 \leq k_j \leq k_{j-1} \leq \dots \leq k_2 \leq n-i}\prod\limits_{u=0}^{k_j}\dfrac{1}{q_u}\prod\limits_{r=1}^{j-2}\prod_{s=k_{j-r+1}}^{k_{j-r}}\dfrac{1}{q_{(T_r+s)}}\prod_{s=k_2}^{n-i}\dfrac{1}{q_{(T_{j-1}+s)}}\nonumber \\
		&=\sum_{j=1}^{[\frac{i}{2}]-1}\sum_{\alpha=2}^{i-2j}\frac{1}{\alpha!}\sum_{\substack{(i_1, i_2, \dots , i_j) \in P(i-\alpha)}} \dfrac{1}{i_1!i_2! \dots i_j!}\sum_{0 \leq k_{j+1} \leq k_{j} \leq \dots \leq k_2 \leq n-i}\prod\limits_{u=0}^{k_{j+1}}\dfrac{1}{q_u}\prod\limits_{r=1}^{j-1}\prod_{s=k_{j-r+2}}^{k_{j-r+1}}\dfrac{1}{q_{(T_r+s)}}\prod_{s=k_2}^{n-i}\dfrac{1}{q_{(i-\alpha+s)}}\nonumber\\
		&=\sum_{j=1}^{[\frac{i}{2}]-1}\sum_{\alpha=2}^{i-2j}\frac{1}{\alpha!}\sum_{\substack{(i_1, i_2, \dots , i_j) \in P(i-\alpha)}} \dfrac{1}{i_1!i_2! \dots i_j!}\sum_{0 \leq k_{j+1} \leq k_{j} \leq \dots \leq k_2 \leq n-i}\prod\limits_{u=0}^{k_{j+1}}\dfrac{1}{q_u}\prod\limits_{r=1}^{j-1}\prod_{s=k_{j-r+2}}^{k_{j-r+1}}\dfrac{1}{q_{(T_r+s)}}\prod_{s=i-\alpha+k_2}^{n-\alpha}\dfrac{1}{q_{s}},
	\end{align*}
	which is same as \eqref{eq6.341}. Therefore our claim is true and hence \eqref{e6.331} follows. Using \eqref{e6.331} in \eqref{e5.332} we get
	\begin{align*}
		&D_{1,n+1}(t)\nonumber \\
		&=t(t+1) \cdots (t+n-1)\Bigg[ \prod_{i=0}^{n-1}\dfrac{1}{ q_i}+ \sum_{i=2}^{n}\frac{1}{i!}\sum_{0  \leq k_1 \leq n-i}\prod\limits_{u=0}^{k_1}\dfrac{1}{q_u}\prod\limits_{v=i+k_1}^{n-1}\dfrac{1}{q_v} +\sum_{i=4}^{n}\mathfrak{U}_{i,2}^{n+1}(t)\Bigg] \nonumber \\
		&=t(t+1) \cdots (t+n-1)\Bigg[ \prod_{i=0}^{n-1}\dfrac{1}{ q_i}+ \sum_{i=2}^{n}\mathfrak{U}_{i,1}^{n+1}(t)\Bigg].
	\end{align*}
	Thus the result is true for the determinant $D_{1,n+1}(t)$. Hence by mathematical induction, the proposition is proved.
\end{proof}

Therefore by Proposition \ref{p5.4}, the inverse of $M(t)$ is given by
\begin{align*}
	M(t)^{-1}	=
	\begin{pmatrix}
		R_{1,1}(t) & 	R_{1,2}(t) & 	R_{1,3}(t) & 	R_{1,4}(t) &  \dots \\
		0 & 	R_{1,1}(t+1) &	R_{1,2}(t+1) & 	R_{1,3}(t+1) &  \dots \\
		0 & 0 & 	R_{1,1}(t+2) & 	R_{1,2}(t+2) & \dots \\
		0 & 0 & 0 &	R_{1,1}(t+3)  & \dots \\
		\vdots & \vdots & \vdots & \vdots& \ddots
	\end{pmatrix},
\end{align*}
where $R_{1,1}(t)=1, R_{1,2}(t)=-\dfrac{t}{q_0}, R_{1,3}(t)=\dfrac{t(t+1)}{q_0(t)}\bigg\{\dfrac{1}{q_1(t)}+\dfrac{1}{2!}\bigg\} ,\\ $ ~$R_{1,4}(t)=\dfrac{t(t+1)(t+2)}{q_0(t)}\bigg\{\dfrac{1}{q_1(t)q_2(t)}+\dfrac{1}{2!}\Big(\dfrac{1}{q_1(t)}+\dfrac{1}{q_2(t)}\Big)+\dfrac{1}{3!}\bigg\}$ and so on.\\
In general, for $n \geq 2$,
\begin{align*}
	R_{1,n}(t)={(-1)}^{n-1}t(t+1)\dots (t+n-2)\mathcal{L}_n(t),
\end{align*} 
where 
$\mathcal{L}_n(t)$ is given by \eqref{eq22}. Consider $\mathcal{L}_1(t)=1$.
Since $M(t)^{-1}$ and $N(t)$ are both upper triangular matrices, each entry in their product contains only a finite sum. Hence the product $M(t)^{-1}N(t)$ is well defined and is again an infinite upper triangular matrix.
Therefore
\begin{align*}
	&M(t)^{-1}N(t)= \begin{pmatrix}
		1 & 	-t \mathcal{L}_2(t)& 	t(t+1) \mathcal{L}_3(t) & -t(t+1)(t+2) \mathcal{L}_4(t) &  \dots \\
		0 & 	1 &	-(t+1) \mathcal{L}_2(t+1) & 	-(t+1)(t+2) \mathcal{L}_3(t+1) &  \dots \\
		0 & 0 & 	1 & 	-(t+2)\mathcal{L}_2(t+2) & \dots \\
		0 & 0 & 0 &	1  & \dots \\
		\vdots & \vdots & \vdots & \vdots  & \ddots
	\end{pmatrix}\nonumber\\
	& \hspace{2cm}	\cdot\begin{pmatrix}
		\dfrac{1}{q_0(t)} & -\dfrac{t}{q_0(t)} & \dfrac{t(t+1)}{2!q_0(t)} & -\dfrac{t(t+1)(t+2)}{3!q_0(t)}  & \dots \\
		0 & \dfrac{1}{q_1(t)} &-\dfrac{t+1}{q_1(t)} & \dfrac{(t+1)(t+2)}{2!q_1(t)}  & \dots \\
		0 & 0 & \dfrac{1}{q_2(t)} & -\dfrac{t+2}{q_2(t)} &  \dots \\
		0 & 0 & 0 & \dfrac{1}{q_3(t)} & \dots\\
		\vdots & \vdots & \vdots & \vdots  & \ddots
	\end{pmatrix} \nonumber\\
	&=\begin{pmatrix}
		H_{1,1}(t) & 	H_{1,2}(t) & 	H_{1,3}(t) & 	H_{1,4}(t) &  \dots \\
		0 & 	H_{1,1}(t+1) &	H_{1,2}(t+1) & 	H_{1,3}(t+1) &  \dots \\
		0 & 0 & 	H_{1,1}(t+2) & 	H_{1,2}(t+2) & \dots \\
		0 & 0 & 0 &	H_{1,1}(t+3)  & \dots \\
		\vdots & \vdots & \vdots & \vdots & \ddots
	\end{pmatrix}=H(t),
\end{align*}
where $H_{1,1}(t)=\dfrac{1}{q_0(t)}$ and for $n \geq 2$
\begin{align}
	\label{eq27}
	H_{1,n}(t)= {(-1)}^{n-1}t(t+1)\dots (t+n-2)\sum_{i=1}^{n}\dfrac{\mathcal{L}_i(t)}{(n-i)!q_{i-1}(t)}.
\end{align}
If we multiply both sides of \eqref{eq19} by $M(s_1)^{-1}$ then L.H.S. of \eqref{eq19} gives $M(s_1)^{-1}N(s_1)V(s_1+s_2, s_3, \dots, s_r)$ and comparing the first argument   one can express the $q$-multiple zeta function of depth $r$ in terms of translates of that of depth $r-1$. But the product $H(s_1)V(s_1+s_2, s_3, \dots, s_r)$ is not well defined, because the formal product of the matrix $H(s_1)$ with the column vector $V(s_1+s_2, s_3, \dots, s_r)$ produces entries which are not convergent series.  
To overcome this difficulty, we fix  an integer $K \geq 1$ and consider the sets
$$I=I_K=\{l \in \mathbb{Z}: 0 \leq l \leq K-1\}, J=J_K= \{l \in \mathbb{Z}: l \geq K\}.$$
Using these index sets, we can write infinite matrices in block form, such as
$$M(s_1)=
\begin{pmatrix}
	M_{II}(s_1) & M_{IJ}(s_1)\\
	0_{JI} & M_{JJ}(s_1)
\end{pmatrix}, ~~
N(s_1)=
\begin{pmatrix}
	N_{II}(s_1) & N_{IJ}(s_1)\\
	0_{JI} & N_{JJ}(s_1)
\end{pmatrix}.$$
From \eqref{eq19}, we can write
\begin{align}
	\label{equation}
	& \hspace{1cm}M_{II}(s_1)V_I(s_1, s_2, \dots, s_r)+M_{IJ}(s_1)V_J(s_1, s_2, \dots s_r)\nonumber \\
	&\hspace{.8cm}=N_{II}(s_1)V_I(s_1+s_2, s_3, \dots, s_r)+N_{IJ}(s_1)V_J(s_1+s_2, s_3, \dots, s_r).
\end{align}
Now $M_{II}(s_1)^{-1}$ is a finite invertible matrix and we have
$$M_{II}(s_1)^{-1}N_{II}(s_1)=H_{II}(s_1).$$
Therefore from \eqref{equation}, we get
\begin{align}
	\label{equ}
	&V_I(s_1, s_2, \dots, s_r)=M_{II}(s_1)^{-1}N_{II}(s_1)V_I(s_1+s_2, s_3, \dots, s_r)+ \nonumber \\
	& \hspace{3cm}M_{II}(s_1)^{-1}N_{IJ}(s_1)V_J(s_1+s_2, s_3, \dots, s_r)-M_{II}(s_1)^{-1}M_{IJ}(s_1)V_J(s_1, s_2,  \dots, s_r).   
\end{align}
As we can see, the first term on the R.H.S of \eqref{equ} is a well defined finite matrix and all the series of meromorphic functions involved in the products of matrices in the second and third terms are normally convergent on all compact subsets of $\mathbb{C}^r$.  
Let us consider
\begin{align*}
	&\hspace{1.1cm}W'_{I}(s_1, s_2,  \dots , s_r)=M_{II}(s_1)^{-1}N_{IJ}(s_1)V_J(s_1+s_2, s_3, \dots, s_r),
\end{align*}
\begin{align*}
	&W_{I}(s_1, s_2, \dots , s_r)= -M_{II}(s_1)^{-1}M_{IJ}(s_1)V_J(s_1, s_2,  \dots, s_r).
\end{align*}
Then \eqref{equ} can be written as
\begin{align}
	\label{eq28}	  V_I(s_1, s_2, \dots, s_r)=H_{II}(s_1)V_I(s_1+s_2, s_3, \dots, s_r)+W'_{I}(s_1, s_2,  \dots , s_r)+W_{I}(s_1, s_2, \dots , s_r).
\end{align} 
Consider $$U_r(K)=\{(s_1, s_2, \dots , s_r) \in \mathbb{C}^r :\text{Re}  (s_1+s_2 +\dots +s_j) >-K, 1 \leq j \leq r \}.$$ 
Since $\text{Re}(s_1+s_2 +\dots +s_j) >-K \implies \text{Re}(s_1+K+s_2+s_3+ \dots +s_j)>0$,  all the entries of the column vector $V_J(s_1, s_2, \dots , s_r)$ are holomorphic on $U_r(K)$.

Now $M_{II}(t)^{-1}N_{IJ}(t)={\Big(R^*_{m,n}(t)\Big)}_{m\leq K, n \in \mathbb{N}}$ and ~~$M_{II}(t)^{-1}M_{IJ}(t)={\Big(R'_{m,n}(t)\Big)}_{m\leq K, n \in \mathbb{N}}$, where
\begin{align*}
	&R^*_{m,n}(t)=R'_{m,n}(t)={(-1)}^{(K+n-m)}(t+m-1)(t+m) \dots (t+K+n-2) \times \nonumber \\
	& \hspace{3.5cm}\sum_{j=1}^{K-m+1}\dfrac{\mathcal{L}_j(t+m-1)}{(K+n-m-j+1)!q_{(m-2+j)}(t)}.
\end{align*}
Therefore   
\begin{align*}
	&R^*_{1,n}(t)=R'_{1,n}(t)\nonumber \\
	&={(-1)}^{(K+n-1)}t(t+1) \dots (t+K+n-2)\bigg(\dfrac{\mathcal{L}_1(t)}{(K+n-1)!q_{0}(t)}+\dfrac{\mathcal{L}_2(t)}{(K+n-2)!q_{1}(t)}+ \dots + \dfrac{\mathcal{L}_{K}(t)}{n!q_{K-1}(t)}\bigg).
\end{align*}

Let  $\xi_K(s_1, s_2, \dots , s_r)$ be the first entry of the column vector $W_{I}(s_1, s_2, \dots , s_r)$ and $\eta_K(s_1, s_2,  \dots , s_r)$ be that of the column vector $W'_{I}(s_1, s_2,  \dots , s_r)$.
Then from \eqref{eq28}, we get 
\begin{align}
	\label{6.6666}
	\z_q(s_1, s_2, \dots , s_r)&=\sum_{i=1}^{K}H_{1,i}(s_1)\z_q(s_1+ s_2+i-1, s_3, \dots , s_r)\nonumber\\
	& + \xi_K(s_1, s_2, \dots , s_r)+\eta_K(s_1,s_2, \dots , s_r),
\end{align}

where $H_{1,n}(s_1)$ is given by \eqref{eq27} and
\begin{align*}
	&	\xi_K(s_1, s_2, \dots , s_r)=\sum_{n=1}^{\infty}R'_{1,n}(s_1)\z_q(s_1+K+n-1, s_2, \dots , s_r), \\
	&  	\eta_K(s_1, s_2, \dots , s_r)=\sum_{n=1}^{\infty}R^*_{1,n}(s_1)\z_q(s_1+s_2+K+n-1, s_3, \dots , s_r).
\end{align*}
\textbf{Note:} After this point, we will adopt all the notations defined in this section wherever necessary.
\section{Poles and Residues}
Here we locate the poles and corresponding residues of the $q$-multiple zeta function. For $1 \leq j \leq r$ and $k \geq 0$, consider the hyperplane $\mathcal{H}_{j, k}$ of $\mathbb{C}^r$ defined by the equation $ s_1+ \dots + s_j=-k$. 

\begin{theorem}
	\label{th3}
	All the poles of $q$-multiple zeta function of depth $r$ are simple and given by the set $\Big\{(s_1, s_2, \dots , s_r) \in \mathbb{C}^r: s_1+s_2 +\dots +s_j \in \mathbb{Z}_{\leq 0}+\dfrac{2\pi i}{\text{log } q}\mathbb{Z}, 1 \leq j \leq r, i=\sqrt{-1}\Big\}.$
\end{theorem}
\begin{proof}
	We prove this theorem by induction on $r$. When $r=1$, then by \cite[Proposition 1]{2}, it is true. Let $r \geq 2$ and assume that the result is true for all $q$-multiple zeta functions of depth $\leq r-1$. 
	For $j \geq 0$ and $1 \leq k \leq r$, consider $$B_{k,j}=\bigg\{(s_1, s_2, \dots , s_r) \in \mathbb{C}^r: s_1+s_2 +\dots +s_k=-j+\dfrac{2\pi i}{\text{log } q}\mathbb{Z}, i=\sqrt{-1}\bigg\}.$$ Then $B_{k, j}$ is disjoint from $U_r(K)$ when $K \leq j$, because $K \leq j$ on $U_r(K)\implies\text{Re}(s_1+s_2+ \dots +s_k)>-K>-j.$
	From \eqref{eq28}, the poles of $\z_q(s_1, s_2, \dots s_r)$ are given by the poles of the first row entry of the column matrix 
	\begin{align*}
		H_{II}(s_1)V_I(s_1+s_2, s_3, \dots, s_r)+W'_{I}(s_1,s_2,  \dots , s_r)+W_{I}(s_1, s_2, \dots , s_r).
	\end{align*} 
	The poles of the first entry of the matrix $H_{II}(s_1)$ are determined by $ q^ {-s_1}-1=0$, which gives $s_1=\dfrac{2\pi ib}{\text{log }q}, b \in \mathbb{Z}$. For the second entry of the first row of the matrix $H_{II}(s_1)$, poles are given by $ (q^ {-s_1}-1)(q^ {-(s_1+1)}-1)=0$, which implies $s_1 \in \bigg\{\dfrac{2\pi ib}{\text{log }q}, -1+ \dfrac{2\pi ib}{\text{log }q} , b \in \mathbb{Z}\bigg\}$. In this way, 
	for the first row of the matrix $H_{II}(s_1)$, we obtain  the collection of simple poles as $$ \bigg\{(s_1, s_2, \dots , s_r) \in \mathbb{C}^r: s_1=-a+\dfrac{2\pi ib}{\text{log } q} , b \in \mathbb{Z}, 0 \leq a \leq K-1\bigg\}.$$
	Therefore entries in the first row of the matrix $H_{II}(s_1)$ are analytic outside the planes given by $$\bigcup_{j=0}^{K-1}\bigg\{(s_1, s_2, \dots , s_r) \in \mathbb{C}^r: s_1=-j+\dfrac{2\pi i}{\text{log } q}\mathbb{Z}\bigg\}=\bigcup_{j=0}^{K-1}B_{1, j}.$$
	Also, each entry of the matrix $V_{I}(s_1+s_2, s_3, \dots , s_r)$ being of depth $r-1$,  all of them  are holomorphic outside the region $\bigcup\limits_{j\geq 0}\{B_{k, j}$, $2 \leq k \leq r\}$.
	Further, entries on the column vector $V_J(s_1, s_2, \dots , s_r)$ and $V_J(s_1+ s_2, s_3, \dots , s_r)$ are holomorphic on $U_r(K)$.
	So for the first row entry of both the matrices $W_{I}(s_1, s_2, \dots , s_r) \text{ and } W'_{I}(s_1, s_2, \dots , s_r)$, the poles are given by
	$$ \{(s_1, s_2, \dots , s_r) \in \mathbb{C}^r: s_1=-a+ \dfrac{2\pi i}{\text{log } q}\mathbb{Z}_{\neq 0}: 0 \leq a \leq K-1\}.$$ 
	Since $\big\{U_r(K): K \geq 1\big\}$ covers $\mathbb{C}^r$, the collection of all the simple poles of $q$-multiple zeta function of depth $r$ is given by   $$\bigg\{(s_1, s_2, \dots , s_r) \in \mathbb{C}^r: s_1+s_2 +\dots +s_j \in \mathbb{Z}_{\leq 0}+\dfrac{2\pi i}{\text{log} q}\mathbb{Z}, 1 \leq j \leq r, i=\sqrt{-1}\bigg\}.$$  
	
	
\end{proof}
\begin{theorem}
	\label{th4}
	The residue of the $q$-multiple zeta function of depth $r$ along the hyperplane $\mathcal{H}_{1,n}$, ~$n \geq 0$, is the restriction of $-n!\mathcal{L}_{n+1}(-n)\dfrac{\z_q(s_2, s_3, \dots, s_r)}{\text{log }q}$ to  $\mathcal{H}_{1,n}$, where 
	$\mathcal{L}_{n+1}(-n)$ is given by \eqref{eq22}.
\end{theorem}
\begin{proof}
	From \eqref{eq28}	We have
	\begin{align*}
		V_I(s_1, s_2, \dots, s_r)=H_{II}(s_1)V_I(s_1+s_2, s_3, \dots, s_r)+W'_{I}(s_1, s_2, \dots , s_r)+W_{I}(s_1, s_2, \dots , s_r).
	\end{align*}
	 The entry on the first row of $W_{I}$ and $W'_{I}$ has no pole along the hyperplane $\mathcal{H}_{1,n}$ for all $ n \geq 0.$
	Also all the entries on the first row of $H_{II}$ contain the term $\dfrac{s_1}{q^{-s_1}-1}$ except the first entry $H_{1,1}(s_1)$. So only $H_{1, 1}(s_1)$ gives simple pole along  $\mathcal{H}_{1,0}$.
	Now consider the function $$\z_q(s_1, \dots, s_r)-H_{1,1}(s_1)\z_q(s_1+s_2, s_3, \dots , s_r),\text{ where }  H_{1,1}(s_1)=\dfrac{1}{q^{-s_1}-1}.$$ It has no pole along $\mathcal{H}_{1,0}$ on $U_r(K)$. Therefore the residue of $\z_q(s_1, \dots, s_r)$ along $\mathcal{H}_{1,0}$ is given by 
	\begin{align*}
		-\dfrac{1}{\text{log } q} \z_q(s_2, s_3, \dots, s_r).
	\end{align*}
	Similarly, only $H_{1, 2}(s_1)$ has simple pole along $\mathcal{H}_{1,1}$. So the function
	 $$\z_q(s_1, \dots, s_r)-H_{1,2}(s_1)\z_q(s_1+s_2+1, s_3, \dots , s_r),\text{ where }  H_{1,2}(s_1)=\dfrac{-s_1 \mathcal{L}_2(s_1)}{q^{-(s_1+1)}-1},$$
	has no pole along $\mathcal{H}_{1,1}$. 
Therefore the residue  of $\z_q(s_1, \dots, s_r)$ along the hyperplane $\mathcal{H}_{1,1}$ is given by
	\begin{align*}
		&-\dfrac{\mathcal{L}_2(-1)}{\text{log} q} \z_q(s_2, s_3, \dots, s_r).
	\end{align*} 
	In general, only $ H_{1, (n+1)}(s_1)$ has simple pole along $\mathcal{H}_{1,n}$ i.e., the function
	$$\z_q(s_1, \dots, s_r)-H_{1,{n+1}}(s_1)\z_q(s_1+s_2+n, s_3, \dots , s_r),$$ where  $$H_{1,{n+1}}(s_1)=\dfrac{{(-1)}^ns_1(s_1+1)\dots (s_1+n-1) \mathcal{L}_{n+1}(s_1)}{q^{-(s_1+n)}-1},$$ 
	has no pole along $\mathcal{H}_{1,n}$.
	Therefore the residue of $\z_q(s_1, \dots, s_r)$ along  $\mathcal{H}_{1,n}$  is given by
	\begin{align*}
		&-\dfrac{n!\mathcal{L}_{n+1}(-n)}{\text{log} q} \z_q(s_2, s_3, \dots, s_r).
	\end{align*}
	Hence the theorem. 
\end{proof}
\begin{theorem}
	The residue of the $q$-multiple zeta function of depth $r$ along the hyperplane $\mathcal{H}_{j ,k}$, where ~$k \geq 0$ and $2 \leq j \leq r$, is the restriction to $ \mathcal{H}_{j,k}$ of the function given by:
	\begin{align*}
		\sum_{s=0}^{k}   -(k-s)!\mathcal{L}_{k-s+1}(s-k)\dfrac{\z_q(s_{j+1},  \dots, s_r)}{\text{log }q} \times (0, s)^{th } \text{ entry of the matrix } \prod_{d=1}^{j-1}H(s_1+s_2+\dots +s_d).
	\end{align*}     
\end{theorem}
\begin{proof}
	Let $K$ be an integer such that $K >k$. Consider $\mathfrak{W}_I(s_1, \dots , s_r)=W'_{I}(s_1, s_2,  \dots , s_r)+W_{I}(s_1, s_2, \dots , s_r)$. Then from \eqref{eq28}  we have 
	\begin{align}
		\label{eq42}
		V_I(s_1, s_2, \dots, s_r)=H_{II}(s_1)V_I(s_1+s_2, s_3, \dots, s_r)+\mathfrak{W}_{I}(s_1, s_2, \dots , s_r).
	\end{align} 
	By iterating the formula \eqref{eq42} $j-1$ times, we get
	\begin{align}
		\label{eqq30}
		&	V_I(s_1, s_2, \dots, s_r)\nonumber \\
		& =\bigg(\prod_{d=1}^{j-1}H_{II}(s_1+s_2+\dots +s_d)\bigg)	V_I(s_1+ \dots +s_j, s_{j+1}, \dots, s_r)+ \mathfrak{W}_{j,I}(s_1, s_2, \dots , s_r),
	\end{align}
	where $\mathfrak{W}_{j, I}(s_1, s_2, \dots , s_r)$ is a column matrix whose entries are the finite sum of the functions in $s_1, \dots , s_{j-1}$ multiplied with the functions which are holomorphic in $U_r(K)$. So they have no pole along the hyperplane $\mathcal{H}_{j, k}$ in $U_r(K)$. 
	Also the entries of $\prod_{d=1}^{j-1}H_{II}(s_1+s_2+\dots +s_d)$ have no pole along the hyperplane $\mathcal{H}_{j, k}$, because they are rational $q$-exponential functions in $s_1, s_2, \dots , s_{j-1}$.
	The first $k+1$ entries of $V_I(s_1+ \dots +s_j, s_{j+1}, \dots, s_r)$ are $\z_q(s_1+ \dots +s_j+s, s_{j+1}, \dots, s_r)$, where $0 \leq s \leq k$ and the remaining entries are $\z_q(s_1+ \dots +s_j+s, s_{j+1}, \dots, s_r)$, where $s >k$.  By Theorem \ref{th3}, $\z_q(s_1+ \dots +s_j+s, s_{j+1}, \dots, s_r)$, where $s >k$, have no pole along the hyperplane $\mathcal{H}_{j, k}$. So only the functions $\z_q(s_1+ \dots +s_j+s, s_{j+1}, \dots, s_r)$, where $0 \leq s \leq k$, have poles along $\mathcal{H}_{j, k}$ in $U_r(K)$.
From equation \eqref{6.6666}, we can write
	\begin{align*}
		&\z_q(s_1+ \dots +s_j+s, s_{j+1}, \dots, s_r)\nonumber \\
		&=\sum_{d=1}^{K}H_{1,d}(s_1+ \dots +s_j+s)\z_q(s_1+ \dots +s_{j+1}+d-1, s_{j+2}, \dots , s_r)\nonumber\\
		& + \xi_K(s_1+ \dots +s_j+s, s_{j+1}, \dots, s_r)+\eta_K(s_1+ \dots +s_{j}+s, s_{j+1}, s_{j+2}, \dots, s_r).
	\end{align*}
	Following the same technique of finding residue in Theorem \ref{th4}, we get the residue of  $\z_q(s_1+ \dots +s_j+s, s_{j+1}, \dots, s_r)$ along $\mathcal{H}_{j, k}$ and is given by
	\begin{align*}
		-(k-s)!\mathcal{L}_{k-s+1}(s-k)\dfrac{\z_q(s_{j+1},  \dots, s_r)}{\text{log }q}.  
	\end{align*}
	Therefore from \eqref{eqq30}, 
	the residue of $\z_q(s_1, s_2, \dots, s_r)$ along the hyperplane $\mathcal{H}_{j, k}$ in $  U_r(K)$ is given by the restriction to $\mathcal{H}_{j, k} \cap U_r(K)$ of the function
	\begin{align*}
		\sum_{s=0}^{k}   -(k-s)!\mathcal{L}_{k-s+1}(s-k)\dfrac{\z_q(s_{j+1},  \dots, s_r)}{\text{log }q} \times (0, s) \text{th entry of the matrix } \prod_{d=1}^{j-1}H(s_1+s_2+\dots +s_d).
	\end{align*} 
	The theorem follows, since for $K>k$, the open sets $U_r(K)$ cover $\c$. 
\end{proof}

\section{Meromorphic continuation of BZ model of $q$-multiple zeta function }
Consider an open subset $\mathcal{U}_r=\{(s_1, \dots , s_r) \in \c: \text{Re}(s_1+\dots +s_d) >d, 1 \leq d \leq r\}$ of $\c$. On this open set $\mathcal{U}_r$, Bradley \cite{11} and  Zhao \cite{1} defined another model of the $q$-multiple zeta function by 
\begin{align*}
	\zeta_q^{\text{BZ}}(s_1,\dots ,s_r) 
	&=\sum_{k_1>\dots >k_r\geq 1}\dfrac{q^{k_1(s_1-1)+ \dots +k_r(s_r-1)}}{[k_1]^{s_1} \dots [k_r]^{s_r}}.
\end{align*} 
This model is called the BZ-model of the $q$-multiple zeta function and as $q \r 1$, it also gives the multiple zeta function. The case $r=1$ was first studied by M. Kaneko et al. in \cite{2}. The analytic continuation of the BZ model was studied by J. Zhao \cite{1} using a different approach, namely, the binomial theorem. Here we give some remarks on how we can find the analytic continuation of this model using a similar process as used in the case of the SZ model.
\begin{remark}
	For any integer $n \geq 1$ and any complex number $s$, consider the difference $$ \mathcal{B}= q^{-(n-1)}{\bigg(\frac{[n]}{q^{n-1}}\bigg)}^{-s} - q^{-n}{\bigg(\frac{[n]}{q^n}+1\bigg)}^{-s}.$$ 
	Then by the Taylor series expansion 
	\begin{align}
		\label{eq5.17}
		\mathcal{B}
		&=q^{-n }{\bigg(\frac{q^{n}}{[n]}\bigg)}^{s}(q^{1-s}-1)+q^{-n } \sum_{k \geq 0} {(-1)}^{k} \frac{s (s+1) \cdots (s + k )}{(k+1)!}{\bigg(\frac{q^n}{[n]}\bigg)}^{s+k+1}.
	\end{align}
	Taking summation over $n$  from $1$ to $\infty$ on both sides of \eqref{eq5.17}, we get
	\begin{align*}
		&1= \sum_{n \geq 1} \bigg[{\frac{q^{n{(s-1)}}}{[n]^{s}}}(q^{1-s}-1)+\sum_{k \geq 0} {(-1)}^{k} \frac{s (s +1)\dots (s + k )}{(k+1)!}{\frac{q^{n(s+k)}}{[n]^{s+k+1}}}\bigg],
	\end{align*}
	which gives the following translation formula for the $q$-analogue of the Riemann zeta function
	\begin{align*}
		&1= (q^{1-s}-1)\z_q^{\text{BZ}}(s) +\sum_{k \geq 0} {(-1)}^{k} \frac{s(s+1) \cdots (s + k )}{(k+1)!}\z_q^{\text{BZ}}(s+k+1),  
	\end{align*}
	where Re$(s)>1$.	
	Now taking summation on both sides of \eqref{eq5.17} over $n=n_1$ from $n_2+1$ to $\infty$ and following the same procedure as that of the SZ model, we get
	\begin{align}
		\label{eq5.22}
		&\sum_{k \geq 0} {(-1)}^{k} \frac{s_1(s_1 +1) \cdots (s_1 + k -1)}{k!}{\frac{q^{n_2(s_1+k-1)}}{[n_2]^{s_1+k}}}\nonumber \\
		&=\sum_{n_1=n_2+1}^{\infty}\bigg[{\frac{q^{n_1{(s_1-1)}}}{[n_1]^{s_1}}}(q^{1-s_1}-1)+\sum_{k \geq 0} {(-1)}^{k} \frac{s_1 (s_1 +1)\cdots (s_1 + k )}{(k+1)!}{\frac{q^{n_1(s_1+k)}}{[n_1]^{s_1+k+1}}}\bigg].
	\end{align}
	Multiplying both sides of \eqref{eq5.22} by  $\prod_{i=2}^{r}{\bigg(\frac{q^{n_i(s_i-1)}}{[n_i]^{s_i}}\bigg)}$ and taking summation over $n_r$ from $1$ to $\infty$, over $n_{r-1}$ from $n_r+1$ to $\infty$ and so on upto over $n_2$ from $n_3+1$ to $\infty$, we get
	\begin{align*}
		&\sum_{n_2 > n_3 > \dots >n_r \geq 1}\sum_{k \geq 0} {(-1)}^{k} \frac{s_1 (s_1 +1) \dots (s_1 + k -1)}{k!}\bigg[\dfrac{q^{{n_2}{(s_1+s_2+k-1)}}q^{n_3(s_3-1)}\dots q^{n_r(s_r-1)}}{{[n_2]}^{s_1+s_2+k} [n_3]^{s_3}\dots [n_r]^{s_r}} \nonumber \\
		& \ \hspace{5cm} +(1-q) \dfrac{q^{{n_2}{(s_1+s_2+k-2)}}q^{n_3(s_3-1)}\dots q^{n_r(s_r-1)}}{{[n_2]}^{s_1+s_2+k-1} [n_3]^{s_3}\dots [n_r]^{s_r}}\bigg]\nonumber\\
		&= (q^{1-s_1}-1)\sum_{n_1 >n_2 > n_3 > \dots >n_r \geq 1}\dfrac{q^{{n_1}{(s_1-1)}}q^{n_2(s_2-1)}\dots q^{n_r(s_r-1)}}{{[n_1]}^{s_1} [n_2]^{s_2}\dots [n_r]^{s_r}} \nonumber \\
		& \ \ + \sum_{n_1 >n_2 > n_3 > \dots >n_r \geq 1}\sum_{k \geq 0} {(-1)}^{k} \frac{s_1 (s_1 +1)\cdots (s_1 + k )}{(k+1)!}\dfrac{q^{{n_1}{(s_1+k)}}q^{n_2(s_2-1)}\dots q^{n_r(s_r-1)}}{{[n_1]}^{s_1+k+1} [n_2]^{s_2}\dots [n_r]^{s_r}}. 
	\end{align*}
	Now following a similar technique of obtaining the translation formula of the SZ model, we get the following translation formula for the BZ model and the corresponding meromorphic continuation.
\end{remark}

\begin{theorem}
	\label{th5}
	Let $r \geq 2$ be any positive integer. Then the $q$-multiple zeta function of depth $r$ extends to a meromorphic function on $\mathbb{C}^r$ satisfying 
	\begin{align}
		\label{eq48}
		&\sum_{k \geq 0} {(-1)}^{k} \frac{s_1 (s_1 +1)\cdots (s_1 + k -1)}{k!}\bigg(\z_q^{\text{BZ}}(s_1+s_2+k,s_3, \dots , s_r)\nonumber \\
		& \ \ +(1-q)\z_q^{\text{BZ}}(s_1+s_2+k-1,s_3, \dots , s_r)\bigg)\nonumber \\
		& =\sum_{k \geq 0}{(-1)}^{k} \frac{s_1(s_1 +1) \cdots (s_1 + k )}{(k+1)!}\z_q^{\text{BZ}}(s_1+k+1, s_2 , \dots , s_r)\nonumber \\
		& \ \ +( q^{-(s_1-1)}-1)\z_q^{\text{BZ}}(s_1,s_2, \dots , s_r),
	\end{align} 	
	where all the series appearing in \eqref{eq48} converge normally on any compact subset of $\mathcal{U}_r$.
\end{theorem}
\begin{remark}
	The translation formula $\eqref{eq48}$ has a matrix representation using which we can locate the poles of the BZ model, which were already located by J. Zhao in \cite{1} using a different approach. The matrix representation of $\eqref{eq48}$ is given by
	\begin{align*}
		\label{eq46}
		A(s_1 )\mathcal{V}(s_1+s_2, s_3, \dots , s_r)=B(s_1)\mathcal{V}(s_1, s_2, \dots, s_r),
	\end{align*}
	where 
	\begin{align*}
		\mathcal{V}(s_1, s_2, \dots, s_r)=	
		\begin{pmatrix}
			\z_q^{BZ}(s_1, s_2, \dots , s_r)
			\vspace{.5em}\\ 
			\vspace{.5em}
			\z_q^{BZ}(s_1+1, s_2, \dots , s_r)\\
			\vspace{.5em}
			\z_q^{BZ}(s_1+2, s_2, \dots , s_r)\\
			\vdots 
		\end{pmatrix},
	\end{align*}
	\begin{align*}
		B(s_1)	=
		\begin{pmatrix}
			1 & \dfrac{s_1}{(q^{-(s_1-1)}-1)} & -\dfrac{s_1(s_1+1)}{2!(q^{-(s_1-1)}-1)} & \dfrac{s_1(s_1+1)(s_1+2)}{3!(q^{-(s_1-1)}-1)}  & \dots \\
			0 & 1 &\dfrac{s_1+1}{(q^{-s_1}-1)} & -\dfrac{(s_1+1)(s_1+2)}{2!(q^{-s_1}-1)} &  \dots \\
			0 & 0 & 1 & \dfrac{s_1+2}{(q^{-(s_1+1)}-1)} & \dots \\
			0 & 0 & 0 & 1 & \dots \\
			\vdots & \vdots & \vdots & \vdots   & \ddots
		\end{pmatrix},
	\end{align*}
	\begin{align*}
		A(s_1)=
		\begin{pmatrix}
			\dfrac{1-q}{(q^{-(s_1-1)}-1)} & -\dfrac{(1-q)s_1-1}{(q^{-(s_1-1)}-1)} & \dfrac{\frac{(1-q)s_1(s_1+1)}{2!}-s_1}{(q^{-(s_1-1)}-1)} &  -\dfrac{\frac{(1-q)s_1(s_1+1)(s_1+2)}{3!}-\frac{s_1(s_1+1)}{2!}}{(q^{-(s_1-1)}-1)}& \dots \\
			0 & \dfrac{1-q}{(q^{-s_1}-1)} & -\dfrac{(1-q)(s_1+1)-1}{(q^{-s_1}-1)} & \dfrac{\frac{(1-q)(s_1+1)(s_1+2)}{2!}-(s_1+1)}{(q^{-s_1}-1)} & \dots \\
			0 & 0 & \dfrac{1-q}{(q^{-(s_1+1)}-1)} & -\dfrac{(1-q)(s_1+2)-1}{(q^{-(s_1+1)}-1)} & \dots \\
			0 & 0 &0 & \dfrac{1-q}{(q^{-(s_1+2)}-1)} &\dots \\
			\vdots & \vdots & \vdots & \vdots & \ddots
		\end{pmatrix}.
	\end{align*} 
\end{remark}
Using the techniques of Theorem \ref{th3}, one can also prove the following theorem:
\begin{theorem}
	All the poles of the BZ model of $q$-multiple zeta function of depth $r$  are simple and are given by the set $\mathfrak{F}'_=\bigg\{(s_1, s_2, \dots , s_r) \in \mathbb{C}^r:  s_1 \in 1+ \dfrac{2\pi i}{\text{log } q}\mathbb{Z} \text{ or } s_1 \in \mathbb{Z}_{\leq 0}+\dfrac{2\pi i}{\text{log }q}\mathbb{Z}_{\neq 0} \text{ or }  s_1+s_2 +\dots +s_j \in \mathbb{Z}_{\leq j}+\dfrac{2\pi i}{\text{log }q}\mathbb{Z}, 2 \leq j \leq r, i=\sqrt{-1}\bigg\}.$
\end{theorem}

\begin{remark}
	A variant of the SZ model of $q$-multiple zeta function, called the $q$-multiple zeta-star function, is defined on $U_r$ by
	\begin{equation*}
		\z_q^{\text{SZ}\star}(s_1,...,s_r)=\sum_{k_1 \geq \dots \geq k_r\geq 1}\dfrac{q^{k_1s_1+ \dots +k_rs_r}}{[k_1]^{s_1} \dots [k_r]^{s_r }}.
	\end{equation*}
	One can find the meromorphic continuation of this function to the whole complex plane and even locate the poles. For that, consider $$A= {\Bigg(\frac{[n]}{q^{n-1}}\Bigg)}^{-s} - {\Bigg(\frac{[n]}{q^n}+1\Bigg)}^{-s},$$	
	where $n$ is a positive integer and $s$ is a complex number. By the Taylor series expansion, we get
	\begin{align}
		\label{eq5.4600}
		{\Bigg(\frac{[n]}{q^{n-1}}\Bigg)}^{-s} - {\Bigg(\frac{[n]}{q^n}+1\Bigg)}^{-s}={\bigg(\frac{q^{n}}{[n]}\bigg)}^{s}(q^{-s}-1)+\sum_{k \geq 0} {(-1)}^{k} \frac{s(s +1) \dots (s + k )}{(k+1)!}{\bigg(\frac{q^n}{[n]}\bigg)}^{s+k+1}.
	\end{align}
	Taking $n=n_{1}$ and summing over $n_1$ from $n_2$ to $\infty$ on both sides of \eqref{eq5.4600}, multiplying by $ \prod_{i=2}^{r}{\bigg(\frac{q^{n_i}}{[n_i]}\bigg)}^{s_i}$ and then taking summation over $n_r$ from $1$ to $\infty$, over $n_{r-1}$ from $n_r$ to $\infty$ and so on up to over $n_2$ from $n_3$ to $\infty$, we get 
	\begin{align*}
		&\sum_{n_2 \geq n_3 \geq \dots \geq n_r \geq 1} \dfrac{q^{{n_2}{(s_1+s_2)}}q^{n_3s_3+ \dots +n_rs_r}}{[n_2]^{s_1+s_2} [n_3]^{s_3}\dots [n_r]^{s_r}} \nonumber \\
		&= (1-q^{s_1})\sum_{n_1 \geq n_2 \geq n_3 \geq \dots \geq n_r \geq 1}\dfrac{q^{n_1s_1+ \dots +n_rs_r}}{{[n_1]}^{s_1}\dots [n_r]^{s_r}} \nonumber \\
		& \ \ + q^{s_1}\sum_{n_1 \geq n_2 \geq  n_3 \geq \dots \geq n_r \geq 1}\sum_{k \geq 0} {(-1)}^{k} \frac{s_1 (s_1 +1)\cdots (s_1 + k )}{(k+1)!}\dfrac{q^{{n_1}{(s_1+k+1)}}q^{{n_2}{s_2}+n_3s_3+ \dots +n_rs_r}}{{[n_1]}^{s_1+k+1}[n_2]^{s_2} [n_3]^{s_3}\dots [n_r]^{s_r}}.
	\end{align*}
	Following the techniques used in the case of SZ model, one can obtain the following theorem:
	\begin{theorem}
		Let $r \geq 2$ be any positive integer. Then the $q$-multiple zeta-star function of depth $r$ extends to a meromorphic function on $\mathbb{C}^r$ satisfying the translation formula
		\begin{align}
			\label{eq5.47}
			&\z_q^{\text{SZ}\star}(s_1+s_2,s_3, \dots , s_r)\nonumber \\
			& =q^{s_1}\sum_{k \geq 0}{(-1)}^{k} \frac{s_1(s_1 +1) \cdots (s_1 + k )}{(k+1)!}\z_q^{\text{SZ}\star}(s_1+k+1, s_2 , \dots , s_r)+(1-q^{s_1})\z_q^{\text{SZ}\star}(s_1,s_2, \dots , s_r),
		\end{align}
		where the series on the R.H.S. converges normally on any compact subset of $U_r$. 
	\end{theorem}
	Further, the matrix representation of the translation formula \eqref{eq5.47} is given by 
	\begin{align*}
		q^{-s_1}V(s_1+s_2, s_3, \dots , s_r)=M(s_1)V(s_1, s_2, \dots, s_r),
	\end{align*}
	where $M(s_1)$ and $ V(s_1, s_2, \dots , s_r)$ are given by \eqref{equn5.16} and  \eqref{equn5.17} respectively.   This can be used to locate the poles of the $q$-multiple zeta-star function.
\end{remark}
\begin{remark}
	Johannes Singer defines a dual of the BZ model \cite{12} on $\mathcal{U}_r$  by
	\begin{align}
		\label{eq64}
		\z_q^D(s_1, \dots , s_r)=\sum_{k_1>\dots >k_r\geq 1}\dfrac{q^{k_1+ \dots +k_r}}{[k_1]^{s_1} \dots [k_r]^{s_r}},
	\end{align}
	where $q>1$.
	For $n \geq 2$ and any complex number $s$, using the difference ~$q^{n-1}{\big([n]-1\big)}^{-s}-q^n{\big(q[n]\big)}^{-s}$, one can derive a translation formula as follows: 
	\begin{align*}
		&\z_q^D(s_1+s_2, s_3, \dots , s_r)- (1-q)\z_q^D(s_1+s_2-1, s_3, \dots , s_r)\\
		& =q^{s_1-1}\sum_{k \geq 0} \frac{s_1(s_1 +1) \dots (s_1 + k )}{(k+1)!}\z_q^D(s_1+k+1, s_2 , \dots , s_r)\nonumber \\  
		& \ \ +( q^{s_1-1}-1)\z_q^D(s_1,s_2, \dots , s_r).
	\end{align*}
	This formula can be used to study the analytic continuation of the function \eqref{eq64}.
	\end{remark}
$\mathbf{Acknowledgement}$: The authors would like to thank Dr. P. Akhilesh, Kerala School of Mathematics, India, for suggesting the work of this article. \\
 



{}

\end{document}